\documentclass[a4paper]{amsart}

\usepackage{amsmath,amssymb,amsthm,stmaryrd}

\usepackage{color,graphicx}
\usepackage[dvipsnames]{xcolor}
\usepackage{accents}
\usepackage{enumerate}
\usepackage{varwidth}

\usepackage{fancyhdr, array,calc,graphicx,url,tabularx}
\usepackage{geometry}
\usepackage{latexsym}
\usepackage{amssymb}
\usepackage{amsmath}

\usepackage{enumerate}
\usepackage{verbatim}
\usepackage{tikz}
\usepackage[colorlinks=true,linkcolor=blue]{hyperref}

\usepackage{changepage}

\usepackage[utf8]{inputenc}
\usepackage{tikz}
\tikzset{
     block/.style={rectangle, draw, fill=red!40, text width=6em,
                   text centered, rounded corners, minimum height=3em},
     arrow/.style={-{Stealth[]}}
     }

\def\undertilde#1{\mathord{\vtop{\ialign{##\crcr
$\hfil\displaystyle{#1}\hfil$\crcr\noalign{\kern1.5pt\nointerlineskip}
$\hfil\tilde{}\hfil$\crcr\noalign{\kern1.5pt}}}}}

{
   \end{minipage}
   \vspace*{\stretch{3}}
   \clearpage
}

\def\undertilde#1{{\baselineskip=0pt\vtop
  {\hbox{$#1$}\hbox{$\scriptscriptstyle\sim$}}}{}}

\newcommand{\utilde}{\undertilde}

\renewcommand{\gg}{\gamma}

\newcommand{\gG}{\Gamma}

\newcommand{\bR}{{\mathbb{R}}}

\newcommand{\rest}{\restriction}

\newcommand{\card}[1]{{\vert #1 \vert} }

\newcommand{\forces}{\Vdash}

\renewcommand{\models}{\vDash}
\newcommand{\powerset}{{\wp}}
\newcommand{\dom}{{\rm dom}}
\newcommand{\rge}{{\rm rge}}
\newcommand{\cp}{{\rm crit }}

\newcommand{\cf}{{\rm cf}}
\newcommand{\lh}{{\rm lh}}

\newtheorem{theorem}{Theorem}[section]

\newtheorem{definition}[theorem]{Definition}

\newtheorem{lemma}[theorem]{Lemma}
\newtheorem{corollary}[theorem]{Corollary}
\newtheorem{claim}[theorem]{Claim}
\newtheorem{subclaim}[theorem]{Subclaim}
\newtheorem{conjecture}[theorem]{Conjecture}

\newtheorem{remark}[theorem]{Remark}
\newtheorem{notation}[theorem]{Notation}
\newtheorem{terminology}[theorem]{Terminology}
\numberwithin{figure}{section}

\newcommand{\rcon}[1]{Conjecture~\ref{#1}}
\newcommand{\rcl}[1]{Claim~\ref{#1}}

\newcommand{\rprop}[1]{Proposition~\ref{#1}}
\newcommand{\rthm}[1]{Theorem~\ref{#1}}
\newcommand{\rlem}[1]{Lemma~\ref{#1}}

\newcommand{\rcor}[1]{Corollary~\ref{#1}}
\newcommand{\rdef}[1]{Definition~\ref{#1}}

\newcommand{\rsec}[1]{Section~\ref{#1}}

\newcommand{\rrem}[1]{Remark~\ref{#1}}

\def\core{{\sf{core}}}

\def\inseg{\trianglelefteq}

\def\k{\kappa}
\def\a{\alpha}
\def\b{\beta}
\def\d{\delta}

\def\l{\lambda}

\def\cop{{\sf{cop}}}

\def\P{{\mathcal{P} }}
\def\W{{\mathcal{W} }}
\def\Q{{\mathcal{ Q}}}
\def\mH{{\mathcal{ H}}}
\def\K{{\mathcal{ K}}}

\def\A{{\mathcal{A}}}
\def\R{{\mathcal R}}

\def\H{{\rm{HOD}}}
\def\M{{\mathcal{M}}}
\def\N{{\mathcal{N}}}

\def\T {{\mathcal{T}}}
\def\U{{\mathcal{U}}}
\def\S{{\mathcal{S}}}
\def\V{{\mathcal{V}}}
\def\X{{\mathcal{X}}}
\def\Y{{\mathcal{Y}}}
\def\Z{{\mathcal{Z}}}

\def\B{{\mathcal{B}}}

\def\card#1{\left|#1\right|}

\def\iff{\mathrel{\leftrightarrow}}

\def\and{\mathrel{\kern1pt\&\kern1pt}}

\def\inseg{\triangleleft}
\def\insegeq{\trianglelefteq}

\def\<#1>{\langle\,#1\,\rangle}

 \input xy
 \xyoption{all}

\begin{document}

\title[Generic generators]{Generic generators}

\author{Grigor Sargsyan} \address{Grigor Sargsyan, IMPAN, Antoniego Abrahama 18, 81-825 Sopot, Poland.}
\email{gsargsyan@impan.pl} 


\begin{abstract}
The goal of this paper is to present an approach to $\sf{Hod\ Pair\ Capturing}$ (${\sf{HPC}}$), which was introduced in \cite[Definition 1.7.1]{SteelCom}. ${\sf{HPC}}$ is one of the main open problems of descriptive inner model theory (see \cite{DIMT}). More specifically, we introduce generic generators, use them to define two principles, the ${\sf{Direct\ Limit\ Independence}}$ (see \rdef{dli}) and the ${\sf{Bounded\ Direct\ Limits}}$ (see \rdef{bdl}), and show that the two principles together imply ${\sf{HPC}}$. 

In a subsequent paper, using the work presented here, we will establish ${\sf{HPC}}$ below a Woodin cardinal that is a limit of Woodin cardinals and derive other applications. It is expected that the concept of generic generators will play a key role in the eventual resolution of ${\sf{HPC}}$. 

\end{abstract}

\maketitle
\setcounter{tocdepth}{1}

Descriptive inner model theory is an area of set theory that uses tools from inner model theory and descriptive set theory to study the internal structure of models of determinacy and also to develop tools to build models of determinacy from various large cardinals. One of its main goals is to create links between various foundational frameworks. The key desirable feature of the links sought by descriptive inner model theory is the ability to transfer set theoretic complexity from one framework to another. 

The most common and the oldest way to measure set theoretic complexity of a foundational framework is by looking at the canonical inner models for large cardinals whose existence is implied by the framework. As is well known, such canonical inner models form a hierarchy that is closely related to the large cardinal hierarchy. Another closely related but more modern method of measuring the set theoretic strength is by examining the kind of universally Baire sets that exist within the framework. Both approaches are based on a profound belief that set theoretic complexity is contained within the canonical inner models for large cardinals as well as the universally Baire sets. The striking aspect of this belief is that these two canonical objects capture all set theoretic complexity of a given set theoretic framework not just a portion of it. 

Descriptive inner model theory has unified both approaches into one single unified method, the ${\sf{Core\ Model\ Induction}}$, and the current work seems to imply that the old pure inner model theoretic approach cannot work, at least not without major modifications. \\

\textbf{The descriptive inner models theoretic approach}\\

As was mentioned above, one of the main goals of descriptive inner model theory is to use tools from inner model theory and descriptive set theory to build set-theoretic-complexity preserving links between various set theoretic frameworks. Since at the time of writing this paper  the canonical inner models that are well-understood are below superstrong cardinals, we must restrict ourselves to this region of the large cardinal hierarchy. We then make the following definition, which is  \cite[Definition 1.74]{SteelCom} and formalizes the content of the previous sentence. 
\begin{definition}[No Long Extender]\label{nle} $\sf{NLE}$ stands for the following statement: there is no active $\omega_1+1$-iterable countable mouse or an lbr mouse $\M$ such that $\M$'s last extender is a long extender.
\end{definition}

\begin{remark} The exact relationship between  $\sf{NLE}$ for pure mice and $\sf{NLE}$ for lbr mice is currently unknown. It is conceivable that $\sf{NLE}$ for lbr mice might be consistent with a failure of ${\sf{NLE}}$ for pure mice. However, in this paper, the distinction between the two versions of $\sf{NLE}$ will not play any role\footnote{We thank Hugh Woodin for pointing out the aforementioned distinction.}. 
\end{remark}
As was mentioned above, the ${\sf{Core\ Model\ Induction}}$ is currently our most robust tool for building models of determinacy from various set theoretic hypothesis. In \cite{ESSealing}, the author and Trang studied its current limitations, and to solve this limitations the author introduced the following conjecture (see \cite[Conjecture 1.13]{ESSealing} and \cite[Conjecture 0.1]{CCM})

\begin{conjecture}[Covering with Chang Models]\label{ccm} Assume $\sf{NLE}$ and suppose there are unboundedly many Woodin cardinals and strong cardinals\footnote{Using technical inner model theoretic concepts, it is possible to state a similar conjecture without assuming the existence of large cardinals.}. Let $\kappa$ be a limit of Woodin cardinals and strong cardinals and such that either $\k$ is a measurable cardinal or $\cf(\k)=\omega$. Then there is a transitive model $M$ of $\sf{ZFC-Powerset}$ such that
\begin{enumerate}
\item ${\sf{Ord}}\cap M=\k^+$,
\item $M$ has a largest cardinal $\nu$,
\item for any generic $g\subseteq Coll(\omega, <\k)$, letting $\mathbb{R}^* = \bigcup_{\alpha<\k} \mathbb{R}^{V[g\cap Coll(\omega,\alpha)]}$ and $\Gamma^*=\{ A^g\cap \bR^*: \exists \a<\k(A\in \Gamma^\infty_{g\cap Coll(\omega,\alpha)})\}$, in $V(\bR^*)$, 
\begin{center}
$L(M, \bigcup_{\a<\nu}\a^\omega, \Gamma^*, \bR^*)\models \sf{AD^+}$.
\end{center}
\item If in addition there is no inner model with a subcompact cardinal then
\begin{center}
 $L(M)\models {\sf{ZFC}}+\powerset(\nu)=\powerset(\nu)^M+\square_\nu$.
 \end{center}
\end{enumerate}
\end{conjecture}
The importance of the conjecture is that assuming it one can show that the ${\sf{Proper\ Forcing\ Axiom}}$ implies that there is an inner models with a subcompcat cardinal (see \cite[Corollary 0.3]{CCM}). Clause 3 of \rcon{ccm} is a generalization of Woodin's Derived Model Theorem: the statement that $L(\Gamma^*, \bR^*)\models {\sf{AD^+}}$ is just the Derived Model Theorem (see \cite{St09DMT}). \cite[Theorem 0.4]{CCM} verifies that the conjecture is true inside hod mice, which is our main evidence for it. The author strongly believes that proving the conjecture is the most fruitful approach to the age old problem that the ${\sf{Proper\ Forcing\ Axiom}}$ implies the existence of an inner model with a superstrong or a subcompact cardinal. 

According to the current methodology, a key step towards proving \rcon{ccm} is the $\sf{HOD\ Problem}$ (see \cite[Question 3.5]{DIMT} and \cite[Chapter 1.4]{SteelCom}), which asks whether the HOD of a model of determinacy that has the form $V=L(\powerset(\bR))$ satisfies ${\sf{GCH}}$. The problem is the main topic of \cite{HMMSC}, \cite{LSA} and \cite{SteelCom}, and the currnet most successful way of attacking it is by showing that assuming ${\sf{AD^+}}+V=L(\powerset(\bR))$, ${\sf{HOD}}\cap V_\Theta$ is a universe of a hod mouse. In \cite{SteelCom}, Steel reduced the problem to $\sf{Hod\ Pair\ Capturing}$ (see \cite[Definition 1.7.1]{SteelCom}).
\begin{definition}[$\sf{Hod\ Pair\ Capturing}$]\label{hpc} Assume ${\sf{AD^{+}}}$ and suppose $A\subseteq \bR$ is a Suslin and co-Suslin set of reals. Then there is a hod pair $(\M, \Sigma)$ such that $A$ is first order definable over the structure $(HC, \Sigma, \in)$. 
\end{definition}
The $L[\vec{E}]$-capturing, $\sf{LEC}$, introduced in \cite[Definition 1.71]{SteelCom} is the same as $\sf{HPC}$ except that we require that the pair $(\M, \Sigma)$ is a pure extender pair (i.e., $\M$ is just an ordinary mouse). It is shown in \cite{SteelCom} that $\sf{LEC}$ implies $\sf{HPC}$. \cite[Theorem 1.73]{SteelCom} states that ${\sf{AD}}_{\bR}+{\sf{HPC}}+V=L(\powerset{\bR})$ implies that ${\sf{HOD}}\cap V_\Theta$ is a universe of a hod mouse, and hence satisfies ${\sf{GCH}}$. Because of this, \cite[Conjecture 1.75 and 1.76]{SteelCom} are the most outstanding theoretical open problems of descriptive inner model theory. We repeat them below.

\begin{conjecture}\label{lec conjecture} Assume ${\sf{AD^{+}+NLE}}$. Then $\sf{LEC}$ holds.
\end{conjecture}

\begin{conjecture}\label{hpc conjecture} Assume ${\sf{AD^{+}+NLE}}$. Then $\sf{HPC}$ holds.
\end{conjecture}

Something very close to both $\sf{LEC}$ and $\sf{HPC}$ were considered by the author in \cite[Theorem 6.1]{HMMSC}, in \cite{LSA} and also in \cite[Conjecture 3.17]{DIMT}, where they were called Generation of Full Pointclasses. In this paper, we will describe a general framework for attacking \rcon{lec conjecture} and \rcon{hpc conjecture}. In a sequel to this paper, we will show how to use this line of attack to prove both conjectures assuming there is no $\omega_1+1$-iterable mouse with a Woodin cardinal that is a limit of Woodin cardinals.

The reason ${\sf{HPC}}$ is relevant to \rcon{ccm} is that $M$ in \rcon{ccm} is essentially a hod mouse. In order to build it, we first need to prove that $\sf{HOD}$ of $L(\Gamma^*, \bR^*)$ is a hod mouse, and then use this fact to \textit{lift} it above $\Gamma^*$. \\

\textbf{The classical approach}\\

The classical approach to the problem of transference of set-theoretic complexity from one framework to another is via the core model $K$ and its parent model, $K^c$. The core model theory as developed in \cite{CMIP} and \cite{JS13} has been a very successful tool for establishing lower bounds. For example, the results of \cite{CMIP} and \cite{JS13} together show that the saturation of the non-stationary ideal requires a Woodin cardinal. The core model theory plays also a crucial role in core model induction, in the internal steps. For example, it is used crucially in Steel's proof that ${\sf{PFA}}$ implies that ${\sf{AD}}^{L(\bR)}$ (see \cite{St05}). The core model theory is successful because of the Covering Theorems, the simplest of which is Jensen's Covering Lemma (see \cite{CMIP}, \cite{CLUW} and \cite{MSch95}). 

The key shortcoming of core model theory is that mild large cardinals, such as Woodin cardinals, imply that the core model doesn't exist, and hence, it cannot have covering properties. In \cite{JSSS}, Jensen-Schimmerling-Schindler-Steel showed that the core model's parent model, $K^c$, which does not posses some of the key properties of $K$ such as generic absoluteness, also has covering properties at sufficiently closed regular cardinals. The advantage of $K^c$ is that because it is not as rigid, the arguments that show the non-existence of $K$ do not apply. The following is one of the key theorems of \cite{JSSS}.
\begin{theorem}[{\cite[Theorem 0.3]{JSSS}}]\label{kc theorem} Let $\kappa\geq \aleph_3$ be a regular cardinal. Assume that $\kappa$ is countably closed in the sense that for all $\eta<\kappa$, $\eta^{\aleph_0}< \kappa$. Suppose that both $\square(\kappa)$ and $\square_\kappa$ fail. If the certified $K^c$ exists in $V^{Coll(\kappa, \kappa)}$ then there is a subcompact cardinal in the certified $K^c$ of $V^{Coll(\kappa, \kappa)}$. 
\end{theorem}

The statement ``$K^c$ exists" says that the construction producing $K^c$, the $K^c$-construction, converges to a model of the desired height. The $K^c$ constructions are constructions much like G\"odel's $L$ and Kunen's $L[\mu]$ that aim to produce a class size model inheriting all large cardinals from $V$. In this constructions, one inductively picks extenders or ultrafilters from $V$ and adds them to the constructibility hierarchy as a predicate. The construction can be very conservative allowing one to pick only the extenders in $V$ that are total and witness strongness embeddings. The $K^c$ construction used in \cite{JSSS} and in \rthm{kc theorem} is the most liberal $K^c$ construction introduced by Mitchell and Schindler in \cite{MitSch}. It essentially allows as many extenders to be added to the model as is possible. It is an abstract theorem of inner model theory that if the countable submodels of models appearing in the $K^c$ construction have $\omega_1+1$-iteration strategies then the $K^c$ construction converges. Given \rthm{kc theorem}, the following then is the key conjecture of the area.

\begin{conjecture}[{\cite[Conjecture 6.5]{OIMT}}] Suppose $\N$ is a premouse occurring in a $K^c$ construction, that $k\leq \omega$, and that $M$ is a countable premouse such that there is a weak
$k$-embedding from $M$ into $C_k(N)$\footnote{This is the $k$th core of $N$.}; then $M$ is $(k, \omega_1, \omega_1+1)$-iterable.
\end{conjecture}
We let $K^c_{\sf{MiSch}}$ be the $K^c$ construction introduced by Mitchell and Schindler in \cite{MitSch}. Thus, $K^c_{{\sf{MiSch}}}$ is the $K^c$ used in \rthm{kc theorem}.\\ 

\textbf{The goal of the paper}\\

The goal of this paper is to present a general approach to proving ${\sf{HPC}}$. The key new concept introduced in this paper is the concept of a \textit{generic generator}. First we make the following definition.

\begin{definition}\label{generated model} Assume ${\sf{AD^+}}$. Suppose $M$ is a transitive model of ${\sf{AD^+}}+{\sf{NLE}}+V=L(\powerset(\bR))$ such that $\bR\subseteq M$. We say $M$ is \textbf{hod pair generated} if there is a hod pair $(\P, \Sigma)$ such that $\Sigma\not \in M$. We then say that $(\P, \Sigma)$ \textbf{generates} $M$ or is a \textbf{generator} for $M$. The meaning of  ``pure extender generated" is defined similarly. 
\end{definition}
In \rdef{generated model}, the pair $(\P, \Sigma)$ generates $M$ because whenever $A\in M\cap \powerset(\bR)$, $A$ is first order definable over $(HC, \Sigma, \in)$. The key difference between generators and generic generators is that generic generators are not countable. Because of this, their comparison theory is harder to develop and we will do it elsewhere. Here, our main goal is to show how generic generators can be used to prove $\sf{LEC}$ and $\sf{HPC}$.

\begin{remark} \normalfont In this paper, all models of ${\sf{AD}}$ are assumed to be models of ${\sf{AD^+}}$. Thus, by writing $V\models {\sf{AD}}_{\mathbb{R}}$ we implicitly assume that $V\models {\sf{AD^+}}$.
\end{remark}

\textbf{A sequel}\\

As was mentioned above, we will need a new comparison theory for generic generators as the comparison theory developed in \cite{SteelCom} does not apply in a way that can be used to establish either ${\sf{DLIH}}$ or $\sf{BDL}$ (see \rdef{dli} and \rdef{bdl}). This will be done in \cite{CWLD}.

The comparison theorem developed in \cite{CWLD} can then be used to show that the hypothesis of the main theorem of \cite[Theorem 1.12]{LaSa21} is weaker than a Woodin cardinal that is a limit of Woodin cardinals. Therefore, \rthm{kc theorem} and \cite[Theorem 1.12]{LaSa21} imply that the Iterability Conjecture for $K^c_{\sf{MiSch}}$ cannot be proven in ${\sf{ZFC}}$. \\\\
\textbf{Acknowledgments.} The author is grateful to Nam Trang, John Steel and Sandra M\"uller for various conversations concerning the topic of this paper. The author is indebted to the referee for a very helpful list of corrections. The idea of proving ${\sf{HPC}}$ and other statements like it by calculating the powerset operation inside hod pair constructions is due to Steel and Woodin. This idea is heavily used in this paper many times (e.g. \rthm{comp of powerset} and \rthm{getting superstrongs}), and it is also heavily used in \cite{ESC}, \cite{HMMSC}, \cite{LSA} and \cite{SteelCom}. We are indebted to the referee for numerous comments that made the paper infintely better.

The author's work is funded by the National Science Center, Poland under the Weave-UNISONO call in the Weave program, registration number UMO-2021/03/Y/ST1/00281.

\section{The largest generated segment}

This paper is based on \cite{SteelCom} and \cite{MPSC}, and it will use the terminology developed there. In particular, our concept of a hod mouse is the one introduced in \cite[Chapter 1.5, Definition 9.2.2]{SteelCom}.

Suppose $V$ is a model of ${\sf{AD^+}}+{\sf{NLE}}+V=L(\powerset(\bR))$. Let $\Delta^V_{gen, hp}$ be the set of all hod pair \textit{generated sets of reals}. More precisely,  $A\in \Delta^V_{gen, hp}$ if and only if $A\subseteq \bR$ and there is a hod pair $(\P, \Sigma)$ such that $\P$ is countable and $A$ is projective in $Code(\Sigma)$\footnote{$Code(\Sigma)$ is the set of reals coding $\Sigma$ via some natural fixed method of coding members of ${\sf{HC}}$ via reals.}. We let $\Delta^V_{gen, pe}$ be the set of all $A\subseteq \bR$ that are generated by a pure extender pair, i.e., there is a pure extender pair $(\P, \Sigma)$ such that $\P$ is countable and $A$ is projective in $Code(\Sigma)$. The following is the main basic theorem known about $\Delta^V_{gen, hp}$.

\begin{definition}\label{theta reg} $\Theta_{{\sf{reg}}}$ is the theory ${\sf{AD}}_{\mathbb{R}}+``\Theta$ is a regular cardinal".
\end{definition}

\begin{theorem}\label{the generated portion} Assume ${\sf{AD^+}}$. If $\Delta_{gen, hp}\not =\powerset(\bR)$ then 
\begin{center}$L(\Delta_{gen, hp}, \bR)\models \Theta_{\sf{reg}}\ \text{and}\  \cf(\Theta^{L(\Delta_{gen, hp}, \bR)})\geq \omega_1.$
\end{center}
\end{theorem}

We will not prove \rthm{the generated portion}, but it is not hard to deduce it from the main theorem of \cite{GappoSarg}. \rthm{existence of gg}, or rather its proof, shows that assuming ${\sf{HPC}}$ fails, the least strong cardinal $\k$ of the hod pair construction $\P$ of a sufficiently strong background unverse is a limit of Woodin cardinals, and that $\Delta$ can be realized as the set of reals of the derived model of some iterate $\Q$ of $\P$ such that if $i:\P\rightarrow \Q$ is the iteration embedding then $i(\k)=\omega_1$. It then follows from the main theorem of \cite{GappoSarg} that $L(\Delta_{gen, hp}, \bR)\models \Theta_{\sf{reg}}$. We do not know if $L(\Delta_{gen, pe}, \bR)\models \Theta_{\sf{reg}}$. 

\section{Generic generators}\label{generic generator}

 We define $Coll(\omega_1, \bR)$ to be the poset consisting of countable functions $p: \omega_1\rightarrow \bR$ such that $\dom(p)\in \omega_1$ ordered by end-extensions. We do not require that $p$ is injective. We then have that if $V\models {\sf{DC}}$ and $g\subseteq Coll(\omega_1, \bR)$ is $V$-generic then $\bR^{V[g]}=\bR^V$ and $g:\omega_1\rightarrow \bR$ is a surjective function with the property that every real is enumerated unboundedly often. The property of $Coll(\omega_1, \bR)$ that we need
is that if $p, q\in Coll(\omega_1, \bR)$ are two conditions such that $\dom(p)=\dom(q)$ and $g\subseteq Coll(\omega_1, \bR)$ is $V$-generic such that $p\insegeq g$ then $g_q\subseteq Coll(\omega_1, \bR)$ defined by 
\begin{center}
$g_q(\a)=\begin{cases}
q(\a)&: \a\in dom(q)\\
g(\a)&:\text{otherwise}
\end{cases}$\end{center}
is also $V$-generic and $V[g]=V[g_q]$. 

The concept of \textit{generic generator} is the key new concept introduced in this paper.

\begin{definition}\label{generic generator def}  Suppose $\Delta\subseteq \powerset(\bR)$ and $g\subseteq Coll(\omega_1, \bR)$ is generic. We say $(\P, \Sigma)\in V[g]$ is a \textbf{pe-generic generator} for $\Delta$ if in $V[g]$,
\begin{enumerate}
\item ${\sf{Ord}}\cap \P=\omega_1$,
\item $(\P, \Sigma)$ is a pure extender pair such that $\Sigma$ acts on $L[\P]$ and $\Sigma$ is an $\omega_2$-strategy,
\item there is a club $C\subseteq \omega_1$ such that for every $\a\in C$, $\P\models ``\a$ is an inaccessible cardinal",
\item for every $\a<\omega_1$, $\Sigma_{\P|\a}\rest HC^V\in V$  and $\Sigma_{\P|\a}$ is the unique extension of $(\Sigma_{\P|\a}\rest HC^V)$ in $V[g]$, 
\item for every $\a<\omega_1$, $Code(\Sigma_{\P|\a}\rest HC^V)\in \Delta$,
\item for every $A\in \Delta$, there is $\a<\omega_1$ such that $A$ is projective in $Code(\Sigma_{\P|\a}\rest HC^V)$\footnote{This means that $A$ is definable via real parameters over the structure $(HC, Code(\Sigma_{\P|\a}), \in)$.}.
\end{enumerate}
Similarly, we say $(\P, \Sigma)$ is a \textbf{hp-generic generator} if $(\P, \Sigma)$ is a hod pair. If the nature of $\P$ is not important then we simply say that $(\P, \Sigma)$ is a \textbf{generic generator}.
\end{definition}

Our goal is to show that if $M$ is not generated then there is a generic generator for it (see \rthm{existence of gg}). However, generic generators may exist even for generated $M$. Below we give one example. Let $\N_{wlw}$ be the minimal sound active lbr premouse satisfying the sentence: ``There is a Woodin cardinal that is a limit of Woodin cardinals"\footnote{Here by \textit{minimal} we mean that no proper initial segment satisfies its defining property.}. We say $``\N_{wlw}$ exists" if there is an active sound lbr premouse $\M$ projecting to $\omega$ such that $\M$ is $\omega_1$-iterable and has a Woodin cardinal that is a limit of Woodin cardinals. 

\begin{theorem}\label{generic generator for dm} Assume ${\sf{AD^+}}$ and that $\N_{wlw}$ exists. Let $\M=\N_{wlw}$ and let $\Sigma$ be its unique $\omega_1$-iteration strategy. Let $M$ be the derived model of $(\M, \Sigma)$\footnote{By the results of \cite{DMATM}, $M$ is independent of genericity iterations and hence, exists in $V$.}. Then it is forced by $Coll(\omega_1, \bR)$ that there is a hp-generic generator for $\Delta=\powerset(\bR)\cap M$.
\end{theorem}
\begin{proof} Because the main purpose of the theorem is to illustrate how generic generators come to exist, we will only outline the proof.  In particular, we will not explain the meaning of ``$M$ be the derived model of $(\M, \Sigma)$". The main issue here is the independence of $M$ from genericity iterations. Steel showed such independence for ordinary mice in \cite{DMATM}, and these methods also work in the context of hod mice as well. We will not use the current theorem in this paper. 

Let $\l$ be the Woodin cardinal of $\M$ that is a limit of Woodin cardinals. Let $\k<\l$ be a $<\l$-strong cardinal of $\M$. Let $g\subseteq Coll(\omega_1, \bR)$ be generic. Working inside $V[g]$ we describe an $\omega_1$-sequence $(\M_\a, \pi_{\a, \b}: \a<\omega_1)$ of iterates of $\M$ such that 
\begin{enumerate}
\item $\M_0=\M$,
\item for every $\a<\omega_1$, $\M_{\a+1}$ is a non-dropping $\Sigma_{\M_\a}$-iterate of $\M_\a$,
\item $\pi_{\a, \a+1}:\M_\a\rightarrow \M_{\a+1}$ is the iteration embedding,
\item for a limit $\b<\omega_1$, $\M_\b$ is the direct limit of $(\M_\a, \pi_{\a, \gg}: \a<\gg<\b)$ and for $\a<\b$, $\pi_{\a, \b}:\M_\a\rightarrow \M_\b$ is the direct limit embedding,
\item letting $\k_\a=\pi_{0, \a}(\k)$,  $\M_{\a+1}=Ult(\M_\a, E_\a)$ for some extender $E_\a$ with $\cp(E_\a)=\k_\a$, 
\item (Generation Condition:) for every $A\in M$ there is $\a<\omega_1$ such that 
$A$ is projective in $Code(\Sigma_{\M_{\a+1}|\lh(E_\a)})$.
\end{enumerate} 
Condition 6 above, which we call the Generation Condition, is the most important condition. Given the sequence we let $\P^+$ be the direct limit of $(\M_\a, \pi_{\a, \b}: \a<\b<\omega_1)$, and set $\P=\P^+|\omega_1$. Let $\Lambda=\Sigma_\P$. Our strategy is to show that $(\P, \Lambda)$ is a generic generator for $M$.  Generation Condition above guarantees that clause 5 of \rdef{generic generator def} holds. The results of \cite{ADRUB} guarantee that $\Lambda$ can be extended to an $\omega_2$-strategy in $V[g]$\footnote{Unfortunately, as pointed out by the referee, \cite[Theorem 2.1]{ADRUB} is proved assuming ${\sf{AD}}_{\mathbb{R}}$, while here we stated our result under ${\sf{AD^+}}$. However, it is not hard to modify the proof of \cite[Theorem 2.1]{ADRUB} to make it work under ${\sf{AD^+}}$. There are number of ways to prove \cite[Theorem 2.1]{ADRUB}. One way is to use the results of \cite{KKMW} where it is shown that $\Theta$ is a limit of cardinals with a strong partition property. It is important to note that the notion of ``strong partition property" used in \cite{KKMW} is stronger than the one used in \cite{St09DMT}, and this stronger version shows that under ${\sf{AD}}_{\mathbb{R}}$, every set of reals is $\k$-homogenously Suslin for every $\k<\Theta$. The proof also shows that under ${\sf{AD^+}}$, every Suslin, co-Suslin set is $\k$-homogenously Suslin for every $\k<\Theta$. In the past, the author was unaware of this discrepancy in the descriptive set theoretic terminology, and didn't immediately see how to generalize the results of \cite{St09DMT} to conclude that in fact every set of reals is $\k$-homogenously Suslin for every $\k<\Theta$. That this is possible was pointed out to the author by Lukas Koschat. One can alternatively redo the proof of \cite[Theorem 2.1]{ADRUB} and show that under ${\sf{AD^+}}$ every Suslin, co-Suslin set of reals is $<\Theta$-uB. The proof presented in  \cite[Theorem 2.1]{ADRUB} uses the Solovay measure on $\powerset_{\omega_1}(\bR)$ which cannot be shown to exist from just ${\sf{AD^+}}$. However, one can use $\Sigma^2_1$-reflection to avoid this issue. Indeed, assuming our $\Sigma$ is not $<\Theta$-uB, we can find some $W$ such that $\powerset(\bR)\cap W\subseteq \utilde{\Delta}^2_1$ and in $W$, $\Sigma$ is not $<\Theta^W$-uB. But now we do have a supercompactness measure that measures subsets of $\powerset_{\omega_1}(\bR)$ that are in $W$. This supercompactness measure can then be used to carry out the proof presented in \cite[Theorem 2.1]{ADRUB}.}

We finish by showing how to arrange the Generation Condition. Because $M\subseteq L(\Sigma, \bR)$ and $M\models {\sf{AD_{\mathbb{R}}}}$ we have that $\Sigma\not \in M$. It follows that in $V[g]$, $\card{\powerset(\bR)^M}=\omega_1$. Fix then an enumeration $(A_\a: \a<\omega_1)\in V[g]$ of $\powerset(\bR)^M$. Our intention is to ``add" each $A_\a$ to $\M_{\a+1}$. 

We have that the derived model of $\M_\a$ as computed by $\Sigma$ is $M$, and the computation of the aforementioned derived model is independent of the particular way the genericity iteration is performed. It follows that we can find an iteration tree $\T_\a$ on $\M_\a$ according to $\Sigma_{\M_\a}$ that is above $\kappa_\a$ and has last model $\M_\a^*$ such that for some $\nu<\pi^\T(\pi_{0, \a}(\l))$, whenever $h\subseteq Coll(\omega, <\nu)$ is $\M_\a^*$-generic, $A$ has a $<\pi^\T(\pi_{0, \a}(\l))$-universally Baire code in $\M_\a^*[h]$. 

Fix now a $\zeta\in (\nu, \pi^\T(\pi_{0, \a}(\l)))$ that is a Woodin cardinal of $\M_\a^*$ and let $E\in \vec{E}^{\M_\a^{*}}$ be such that $\cp(E)=\k_\a$ and $\lh(E)>(\zeta^+)^{\M_\a^*}$. Set $\M_{\a+1}=Ult(\M_\a, E)$. Then whenever $h\subseteq Coll(\omega, (\zeta^+)^{\M_{\a+1}})$ is $\M_{\a+1}$-generic, $A$ has a $<\pi_{E}^{\M_\a}(\pi_{0, \a}(\l))$-universally Baire code in $\M_{\a+1}[h]$. This is because
 $A\in \utilde{\Delta}^1_1(Code(\Sigma_{\M_{\a+1}|\zeta}))$. 
\end{proof}

\begin{remark} \normalfont
    Assume that the minimal active mouse with a Woodin cardinal that is a limit of Woodin cardinals exists and let $\M$ be this mouse and $\Sigma$ be its strategy. Then the proof of \ref{generic generator for dm} can be used to obtain a pe-generic generator for the derived model of $\M$.
\end{remark}

\section{Production of generic generators}

We will consider fully backgrounded constructions of sufficiently strong backgrounds. \cite[Chapter 4.1]{LSA} introduces all the necessary terminology. \cite[Definition 4.1.8]{LSA} defines the meaning of Suslin, co-Suslin capturing. \cite[Theorem 4.1.12]{LSA}, which is due to Woodin, is our main source of backgrounds. \cite{ADRUB} outlines a proof of \cite[Theorem 4.1.12]{LSA}. 

Suppose $V$ is a model of ${\sf{AD_{\mathbb{R}}}}+{\sf{NLE}}+\neg{\sf{HPC}}+V=L(\powerset(\bR))$. Set $\Delta=\Delta^V_{gen}$. Because we are assuming ${\sf{AD_{\bR}}}$ and $\Delta\not =\powerset(\bR)$, we have that there are good pointclasses beyond $\Delta$. Suppose that $(\mathbb{M}, (P, \Psi), \Gamma^*, A)$ Suslin, co-Suslin captures some good pointclass that is Wadge beyond $\Delta$. It is then not unreasonable to hope that the hod pair construction of $\mathbb{M}$ reaches a hod pair beyond $\Delta$. This has been the main strategy for establishing ${\sf{HPC}}$ and other statements like it, and was successfully implemented both in \cite{HMMSC} and in \cite{LSA}. However, \cite{HMMSC} and \cite{LSA} assume more than just ${\sf{NLE}}$: the first assumes no hod mouse with an inaccessible limit of Woodins and the second no transitive inner model of ${\sf{LSA}}$, both of which are stronger assumptions than ${\sf{NLE}}$. We are unable to prove, assuming just ${\sf{NLE}}$, that the hod pair construction of $\mathbb{M}$ reaches a hod pair beyond $\Delta$, but our failure bears some fruits as it produces generic generators.


 


In this paper, the expression ``some fully backgrounded construction" refers to a $K^c$ construction in which the extenders used are total extenders. The constructions we have in mind may produce pure premice or lbr premice or anything else. In addition, if the background universe has a distinguished extender sequence then we assume that all background extenders come from that extender sequence. There are many types of such fully backgrounded constructions. For example, the reader may wish to consult \cite[Definition 3.1.1]{SteelCom} and also \cite[Definition 3.2]{FarmBC}. These two definitions are slightly different: they impose different initial segment conditions on the models built by backgrounded constructions. 

When we say that ``$\M$ is a model appearing on some fully backgrounded construction" we mean that $\M$ is one of the models constructed by the named construction. Also, if $\M$ is a mouse or an lbr mouse and $\phi$ is a large cardinal property then $M\models \phi[\kappa]$ means that in $\M$, $\k$ has the large cardinal property $\phi$ as witnessed by the extender sequence of $\M$. 

 \begin{lemma}\label{strongs in the construction} Suppose $\d$ is an inaccessible cardinal and suppose $\M$ is a model appearing on some fully backgrounded construction of $V_\d$. Suppose\footnote{$\rho_{\omega}(\M)=\rho(\M)$ is the eventual projectum of $\M$.} $\k<\rho_{\omega}(\M)=\rho(\M)$ is such that $\M\models ``\k$ is a measurable cardinal". Then there is no total extender $E$ on the sequence of $\M$ such that $\cp(E)=\k$ and the resurrection\footnote{The resurrection of $E$ is an extender $F$ that appears in some model of the construction and eventually cores down to $E$. Moreover, $F$ is a restriction of a background extender. See \cite[Chapter 12]{FSIT} or \cite[Chapter 3.1]{SteelCom}.} of $E$ has critical point $>\k$. 
 \end{lemma}
 \begin{proof} Suppose not. Let $\k$ be a measurable cardinal of $\M$ and let $E$ be a total extender on the sequence of $\M$ with $\cp(E)=\k$ such that letting $F$ be its resurrection, $\cp(F)=_{def}\l>\k$. Let $\xi$ be the stage of the backgrounded construction where $F$ is added to the construction. Thus, $\N_\xi=(\M_\xi, G)$ where $G=F\cap \M_\xi$. 
 
 Because $G$ collapses down to $E$ via core embeddings, we must have a strictly increasing sequence of $(\xi_i: i\leq n)$, $(G_i: i\leq n)$ and $(\pi_i: i\leq n)$ such that setting $G_{-1}=G$,
 \begin{enumerate}
\item $\xi_0=\xi+1$,
 \item $\pi_{i}: \M_{\xi_i}\rightarrow \N_{\xi_i-1}$ is the inverse of the core embedding,
 \item $G_{i+1}=\pi_{i+1}^{-1}(G_i)$ for all $i\in [-1,n-1]$,
 \item for all $i\in [0,n-1]$, $\xi_{i+1}$ is the least $>\xi_i$ such that $\M_{\xi_i}\insegeq \N_{\xi_{i+1}-1}$ and $\rho(\N_{\xi_{i+1}-1})<\lh(G_i)$, and
 \item $\rho(\N_{\xi_{n}-1})=\k$. 
 \end{enumerate}
 We then have that $E=\pi_n^{-1}(G_{n-1})$. However, $rud(\M_{\xi_n})\models ``(\k^+)^{\M_{\xi_n}}$ is not a cardinal". It follows that $E$ cannot be total over $\M$. 
 \end{proof}

\begin{corollary}\label{cor to strongs in the construction} Suppose $V$ is a premouse or an lbr premouse, $\d$ is an inaccessible cardinal and $\M$ is a model appearing on some fully backgrounded construction of $V_\d$. Suppose ${\sf{Ord}}\cap \M=\d$ and $\k$ is such that $\M\models ``\k$ is a strong  cardinal". Then $V_\d\models ``\k$ is a strong cardinal". 
  \end{corollary}
  \begin{proof} It follows from \rlem{strongs in the construction} that the resurrections of all total extenders on the sequence of $\M$ with critical point $\k$ must have a critical point $\k$. It follows that for each $\nu$ there is an extender $E$ on the sequence of $V$ with critical point $\kappa$ and $\lh(E)>\nu$. Hence, $\k$ must be a strong cardinal of $V_\d$.
  \end{proof}

\rthm{existence of gg} is the main theorem of this section. 

\begin{notation} $\cop(\T)$ is the common part of $\T$ and is customarily denoted by $\M(\T)$. Since $\M$ is overused in inner model theory, we will use $\cop$ notation.
\end{notation}

We will also use the following lemma, which is the restatement of \cite[Theorem 4.1.12]{LSA}. Recall from \cite{Mo09} that if $\Gamma$ is a lightface pointclass then $\utilde{\Gamma}$ is its boldface version.

\begin{theorem}[Woodin, Theorem 10.3 of \cite{DMATM}]\label{n*x} Assume $ \sf{AD}^+$. Suppose $\Gamma$ is a good pointclass and there is a good pointclass $\gG^*$ such that $\Gamma\subseteq \utilde{\Delta}_{\gG^*}$. Suppose $(N, \Psi)$ is $\gG^*$-Woodin which Suslin, co-Suslin captures some $A\in lub(\Gamma)$. There is then a function $F$ defined on $\mathbb{R}$ such that for a Turing cone of $x$, $F(x)=(\N^*_x, \M_x, \d_x, \Sigma_x)$ is such that
\begin{enumerate}
\item $N\in L_1[x]$,
\item $\N^*_x|\d_x=\M_x| \d_x$,
\item $\M_x$ is a $\Psi$-mouse over $x$: in fact, $\M_x=\M_1^{\Psi, \#}(x)|\k_x$ where $\k_x$ is the least inaccessible cardinal of $\M_1^{\Psi, \#}(x)$ that is $>\d_x$,
\item $\N^*_x\models ``\d_x$ is the only Woodin cardinal",
\item $\Sigma_x$ is the unique iteration strategy of $\M_x$,
\item $\N^*_x= L(\M_x, \Lambda)$ where $\Lambda=\Sigma_x\rest \dom(\Lambda)$ and 
\begin{center}
$\dom(\Lambda)=\{\T\in \M_x: \T$ is a normal iteration tree on $\M_x$, $\lh(\T)$ is a limit ordinal and $\T$ is below $\d_x\}$,
\end{center}
\item setting $\vec{G}=\{(\a, \vec{E}^{\N^*_x}(\a)): \N^*_x\models  ``\lh(\vec{E}^{\N^*_x}(\a))$ is an inaccessible cardinal $<\d_x"\}$ and $\mathbb{M}_x=(\N^*_x, \d_x, \vec{G}, \Sigma_x)$, $(\mathbb{M}_x, (N, \Psi), \Gamma^*, A)$ Suslin, co-Suslin captures $\Gamma$\footnote{Hence, $(\N^*_x, \d_x, \vec{G}, \Sigma_x)$ is a self-capturing background.}. 
\end{enumerate}
\end{theorem}

\begin{remark}
    We assume that the $\N^*_x$ model introduced in \rthm{n*x} is in ms-indexing introduced in \cite{FSIT}. This will play a crucial role in the next proof. 
\end{remark}

In the proof of the next theorem,  ${\sf{NLE}}$ is used to gurantee universality (see \cite[Chapter 8.1]{SteelCom}). 

\begin{theorem}\label{existence of gg} Suppose $V\models {\sf{AD_{\mathbb{R}}}}+{\sf{NLE}}+\neg{\sf{LEC}}$ and set $\Delta=\Delta^V_{gen, pe}$. Let $g\subseteq Coll(\omega_1, \bR)$. Then there is a pe-generic generator for $\Delta$ in $V[g]$.

Similarly, if $V\models {\sf{AD_{\mathbb{R}}}}+{\sf{NLE}}+\neg{\sf{HPC}}$ then there is a hp-generic generator for $\Delta_{gen, hp}$ in $V[g]$.
\end{theorem}
\begin{proof} The proof below produces a pe-generic generators. To produce hp-generic generators we consider hod pair constructions instead of pure $L[\vec{E}]$ constructions.

Because $\Delta\not =\powerset(\bR)$, we can fix $A^*\subseteq \bR$ such that $A^*\not \in \Delta$ and $w(A^*)=w(\Delta)$\footnote{For $B\subseteq \bR$, $w(B)$ is the Wadge rank of $B$ and for a pointclass $\Gamma$, $w(\Gamma)$ is the supremum of $w(B)$ for $B\in \Gamma$.}. Because ${\sf{AD_{\mathbb{R}}}}$ holds, we can fix a good pointclass $\Gamma$ such that all projective in $A^*$ sets of reals are contained in $\undertilde{\Delta}_\Gamma$. Let the tuple $(F, \Gamma^*, (N, \Psi), A)$ be as in \rthm{n*x} applied to $\Gamma$. 
Fix $x_0$ in the cone of \rthm{n*x} and set $M=\N^*_{x_0}$, $\Sigma'=\Sigma_{x_0}$ and $\d=\d_{x_0}$. Let $g\subseteq Coll(\omega_1, \bR)$ be $V$-generic. 
Working in $V[g]$, let $(y_\a:\a<\omega_1)$ be a sequence of reals such that
\begin{enumerate}
\item $y_0=x_0$,
\item for every $z\in \bR$ there is $\a<\omega_1$ such that $z\leq_T y_\a$, and
\item for every $\a<\omega_1$, $(\N^*_{y_\b}: \b<\a)\in HC^{\N^*_{y_\a}}$.
\end{enumerate}
Let $\N$ be the output of the pure $L[\vec{E}]$-construction of $M|\d$ as defined in \cite[Chapter 3.1]{SteelCom}. Let $\kappa$ be any strong cardinal of $\N$. It follows from \rcor{cor to strongs in the construction} that $\k$ is a strong cardinal of $M|\d$. 

We now build an $\omega_1$-sequence $(M_\a, \k_\a, j_{\a, \b}:\a\leq \omega_1)$ of iterates of $M$ such that the following conditions are satisfied.  
\begin{enumerate}
\item $M_0=M$ and $\k_0=\k$,
\item for $\a<\omega_1$,  $\k_\a=j_{0, \a}(\k)$, 
\item for $\a<\omega_1$, $\cp(j_{\a, \a+1})=\k_\a$, 
\item for all limit $\a\leq \omega_1$, $M_\a$ is the direct limit of $(M_\b, j_{\b, \gg}: \b<\gg<\a)$, and
\item for all $\b<\omega_1$, $(M_\a:\a<\b)\in HC^{\N^*_{y_\b}}$.
\end{enumerate}
The construction is very similar to the construction introduced in the proof of \cite[Theorem 9.3.2]{LSA} (see the construction  after Claim 3)\footnote{The author discovered this construction sometime during 2010-2012.}.

It is enough to explain how we build $M_{\a+1}$ given $M_\a$. We have that for each $\b<\omega_1$, $\N^*_{y_\b}|\d_{y_\b}=\M_1^{\Psi}(y_\b)|\d_{y_\b}$\footnote{$\M_1^{\Psi}$ is the minimal class size $\Psi$-mouse with a Woodin cardinal.}. Also, $\Sigma_{y_\b}$ acts on trees that are based on $\N^*_{y_\b}|\d_{y_\b}$ as there are no more extenders in $\N^*_{y_\b}$. The following lemma is the generalization of the corresponding lemma for $\M_1$. 

\begin{lemma} Suppose $\a<\b<\omega_1$. Then $\Sigma_{y_\a}\rest (\N^*_{y_\b}|\d_{y_\b})\in \N^*_{y_\b}$.
\end{lemma}
\begin{proof} Given a transitive set $b$ such that $L_1(b)\models ``b$ is well-ordered", we let $b^{\Psi, \#}$ be the minimal active sound $\Psi$-mouse over $b$. Let $G$ be the function given by $G(b)=b^{\Psi, \#}$. Let $G'$ be the function given by $G'(\U)=(\cop(\U))^{\Psi, \#}$ where $\U$ is any tree on $\N^*_{y_\a}|\d_{y_\a}$ of limit length. Notice that\\\\
(1) $G\rest (\N^*_{y_\b}|\d_{y_\b})$ is definable over  $\N^*_{y_\b}|\d_{y_\b}$, and therefore,\\
(2)  $G'\rest (\N^*_{y_\b}|\d_{y_\b})$ is definable over  $\N^*_{y_\b}|\d_{y_\b}$. \\\\
But $G'$ determines $\Sigma_{y_\a}$. Hence, $\Sigma_{y_\a}\rest (\N^*_{y_\b}|\d_{y_\b})$ is definable over $\N_{y_\b}$.
\end{proof}

Fix now $\a<\omega_1$. We want to build $M_{\a+1}$. Let $y=y_{\a+1}$. Clause 5 above implies that $M_\a\in HC^{\N^*_{y}}$. We now perform a backgrounded construction over $M_\a|(\kappa_\a^+)^{M_\a}$ that builds an iterate of $M_\a$ via a normal tree $\T$ that only uses extenders with critical point $\geq \kappa_\a$.  Set $W=Ult(\M_\a, E)$ where $E\in \vec{E}^{M_\a}$ is such that $\cp(E)=\k_\a$ and $Ult(M_\a, E)\models ``\k_\a$ is not a measurable cardinal". We let $\Lambda=(\Sigma')_W$ and $\d^W$ be the Woodin cardinal of $W$. 

The construction consists of models $(P_\xi, R_\xi: \xi<\d_{y})$. We start by setting $P_0=M_\a|(\k_\a^+)^{M_\a}=W|(\k_\a^+)^W$. The following is our inductive hypothesis:\\\\
IH$_\xi:$ In the comparison of $W$ with $P_\xi$, no disagreement is caused by a predicate indexed on the $P_\xi$ side.\\\\
A consequence of IH$_\xi$ is that in the aforementioned comparison, $P_\xi$ side doesn't move. Let then $\T_\xi$ be the $W$-to-$P_\xi$ tree according to $\Lambda$ with last model $W_\xi$ such that either $P_\xi\insegeq W_\xi$ or $W_\xi\insegeq P_\xi$. Our construction will ensure that $W_\xi\inseg P_\xi$ is impossible for $\xi<\d_{y_\b}$. Assume then that $P_\xi\insegeq W_\xi$. Another consequence of IH$_\xi$ is that for $\xi<\d_{y_\b}$, $P_\xi$ and $R_\xi$ are iterable as $\Psi$-mice and hence, we do not need to worry about the convergence of the construction. We stop the construction if either we build $(P_\xi, R_\xi: \xi<\d_{y})$ or encounter $\xi<\d_{y}$ for which IH$_\xi$ fails. \\\\
\textbf{The successor step.}\\

Suppose we have built $(P_\zeta, R_\zeta: \zeta<\xi)$ and $P_\xi$.  If $\xi=\d_y$ then we stop the construction and declare $P_\xi$ as its output. We define $R_\xi$ and $P_{\xi+1}$ as follows. \\\\
\textbf{1. Adding a Backgrounded Extender.} Suppose $G\in \vec{E}^{\N^*_{y}}$ is such that $(P_\xi, G\cap P_\xi, \in)$ is an ms-premouse\footnote{I.e. the indexing scheme and the initial segment condition are as in \cite{FSIT} and \cite{OIMT}. Our backgrounding condition is basically that of \cite[Chapter 11]{FSIT}.}. Letting $G$ be the least such extender, we set $R_\xi=(P_\xi, G\cap P_\xi, \in)$ and $P_{\xi+1}=\core(R_\xi)$. \\\\
\textbf{2. Adding an Extender from $W_\xi$.} Assume  clause 1 fails. Suppose that there is no extender as above but there is an extender $G\in \vec{E}^{W_\xi}$ such that $\cp(E)<\k_\a$ and $(P_\xi, E, \in)$ is a $\Psi$-premouse. We then let $R_\xi=(\P_\xi, E, \in)$ and $P_{\xi+1}=\core(R_\xi)$.\\\\
\textbf{3. Adding a Branch from $W_\xi$.} Assume clauses 1 and 2 fail, $P_\xi=W_\xi|\nu$, and in $W_\xi$,  there is a branch $b$ indexed at $\nu$. We then let $\R_{\xi+1}=W||\nu$ and $P_{\xi+1}=\core(R_\xi)$. \\\\
\textbf{4. Closing Under Constructibility.} Assume clauses 1-3 fail. Set $R_\xi=rud(P_\xi)$ and $P_{\xi+1}=\core(R_\xi)$.\\\\
\textbf{The limit step.}\\

We define the limit step according to \cite[Definition 6.3]{OIMT}. Suppose we have build $(P_\zeta, R_\zeta: \zeta<\xi)$ and $\xi$ is a limit ordinal. We then let $\P_\xi$ be the $limsup_{\zeta\rightarrow \xi}P_\zeta$. More precisely, $P_\xi$ is defined on ordinals $\b$ such that for some $\zeta<\xi$, $\omega\b\leq {\sf{Ord}}\cap \P_\zeta$. Suppose now that for every $\a<\b$, we have defined $P_\xi||\a$.
Then because for every $\zeta<\xi$, $P_\zeta$ and $R_\zeta$ are iterable, it follows that for some $\zeta^*<\xi$, for all $\zeta\in [\zeta^*, \xi)$, $P_\zeta||\b=P_{\zeta^*}||\b$. We then set $P_\xi||\b=P_\zeta||\b$.

Notice now that if $\l$ is a cardinal of $\P_\xi$ then $(\l^+)^{\P_\xi}$ exists if and only if for some $\zeta^*<\xi$, for all $\zeta\in [\zeta^*, \xi)$, $\P_\xi||(\l^+)^{\P_{\xi}}=\P_\zeta||(\l^+)^{\P_\xi}$. Equivalently, $\l$ is the largest cardinal of $\P_\xi$ if and only if for cofinally many $\zeta<\xi$, $\rho(\P_\zeta)=\l$. 

\begin{lemma}\label{never breaks down} The above construction lasts $\d_y$ many steps.  
\end{lemma}
\begin{proof} Towards a contradiction assume that the construction breaks down at some stage $\xi$. This means that our inductive hypothesis IH$_\xi$ fails. A careful analysis of the construction shows that $\xi$ must be a limit ordinal. Below we provide the details of this analysis. \\

\textit{Claim.} $\xi$ is a limit ordinal.
\begin{proof} Towards a contradiction assume that $\xi=\gg+1$. Recall the tree $\T_\gg$ on $W$ that is according to $\Lambda$. $\T_\gg$ has a last model $W_\gg$ and $P_\gg\insegeq W_\gg$. We now examine how $R_\gg$ is defined. 

Suppose that $R_\gg$ is obtained from $P_{\gg}$ by adding a predicate from $W_\gg$. We examine the case when the added predicate is an extender, the other case being very similar. This means that if $\nu={\sf{Ord}}\cap P_\gg$ then $E_\nu^{W_\gg}\not=\emptyset$ and $R_\gg=(P_\gg, E_\nu, \in)$. If $\rho(R_\gg)=\nu$ or $R_\gg$ is sound then $P_{\gg+1}=P_\xi=R_\gg$. It follows that $P_{\gg+1}\insegeq W_\gg$. Therefore, IH$_\xi$ holds as witnessed by $\T_\gg$. 

Suppose now that $\rho(R_\gg)<\nu$ and $R_\gg$ is not sound. Because we have that $R_\gg\insegeq W_\gg$ and all proper initial segments of $W_\gg$ are sound, we have that $W_\gg=R_\gg$. It then follows that $P_{\xi}=\core(R_\gg)$ is a node on $\T_\gg$. Then $\T_{\leq P_\xi}$ witnesses that IH$_\xi$. 

Suppose now that $R_\gg$ is obtained from $P_\gg$ by adding a backgrounded extender. Let $G\in \vec{E}^{\N^*_y}$ be such that $R_{\gg}=(P_\gg, G\cap P_\gg, \in)$. It follows from the stationarity of the backgrounded constructions that $G\cap W_\gg\in \vec{E}^{W_\gg}$ (e.g. see \cite[Lemma 2.11]{HMMSC} or \cite[Lemma 8.1.2]{SteelCom}). The rest follows exactly the same argument as above.  

It remains to examine the case when $R_\gg=rud(P_\gg)$. If $P_\gg\inseg W_\gg$ then the proof is exactly as above. Assume then $P_\gg=W_\gg$. Because $W_\gg$ is sound, we must have that the main branch of $\T_\gg$ doesn't drop. Let $\l=\pi^{\T_\gg}(\d^W)$. We have that $\l<\d_y$. $S$-constructions now show that $(\N^*_{y}|\l)^{ \Psi, \#}\models ``\l$ is a Woodin cardinal"\footnote{Because every extender on $W_\gg$ with critical point $>\k_\a$ is backgrounded, we must have that $\N^*_{y}|\l$ is generic over $W_\gg$. E.g. see \cite[Lemma 3.40]{HMMSC} or \cite[Lemma 1.3]{NegRes}. $(\N^*_{y}|\l)^{ \Psi, \#}$ is the minimal active $\Psi$-mouse over $\N^*_y|\l$. Thus, for some $\a$, the universe of $(\N^*_{y}|\l)^{ \Psi, \#}$ is $\mathcal{J}^\Psi_\a[\N^*_y|\l]$. To show that $\l$ is Woodin in $(\N^*_{y}|\l)^{ \Psi, \#}$ we use the fact mentioned in the first sentence of this footnote. It shows that every $f:\l\rightarrow \l$ which is in $(\N^*_{y}|\l)^{ \Psi, \#}$ is dominated by some $g:\l\rightarrow \l$ such that $g\in W_\gg$. We can then pick an $F\in \vec{E}^{\N^*_y}$ such that $G=_{def}F\cap W_\gg$ witnesses Woodinness of $\l$ for $g$ in $W_\gg$. It is then easy to see that $F$ witnesses Woodinness of $\l$ for $f$ in $(\N^*_{y}|\l)^{ \Psi, \#}$. For more details see the Claim in the proof of \cite[lemma 1.3]{NegRes}. This type of arguments were first used by Steel and Woodin.}. This contradicts the minimality of $\d_y$.
\end{proof}

We thus have that $\xi$ is a limit ordinal. We now define the $W$-to-$P_\xi$ tree. For each $P_\xi$ cardinal $\l$ such that $(\l^+)^{P_\xi}<{\sf{Ord}}\cap P_\xi$, let $\U_\l$ be the $W$-to-$P_\xi|(\l^+)^{P_\xi}$ tree according to $\Lambda$. We have that $P_\xi|(\l^+)^{P_\xi}$ is a model in our construction; thus, in fact $\U_\l=\T_\zeta$ for some $\zeta$. Let $S_\l$ be the last model of $\U_\l$. Because $\U_\l$ is just a comparison tree, it follows that for $\l<\l^*$, $\U_\l\inseg \U_\l^*$. Thus, if $P_\xi$ doesn't have a largest cardinal then setting $\T_\xi=_{def}\cup_{\l}\U_\l$, $\T_\xi$ is as desired. 

Suppose then $\l$ is the largest cardinal of $P_\xi$. We are still in the process of defiing $W$-to-$P_\xi$ tree.  Let $\zeta<\xi$ be such that $\l={\sf{Ord}}\cap \P_\zeta$ and $\P_\zeta=\P_\xi|\l$. Because $\l$ is a cardinal of $P_\xi$, we have that for every $\zeta^*\in [\zeta, \xi)$, $\rho(P_{\zeta^*})\geq \l$. Because $\l$ is the largest cardinal of $P_\xi$ and $\xi$ is a limit ordinal, we have that for cofinally many $\theta<\xi$, $P_\theta$ is sound and $\rho(P_\theta)=\l$. Fix one such $\theta>\zeta$. We claim that $P_\theta\insegeq W_\zeta$. Towards a contradiction assume not. By minimizing such a $\theta$, we can assume that $\theta=\gg+1$ and $P_\theta=\core(R_\gg)$. Because $R_\gg$ is not sound, we must have that $W_\gg=R_\gg$. Because $R_\gg|\l=W_\zeta|\l$, we must have that $\T_\zeta\insegeq \T_\gg$. Because $R_\gg$ is not sound, we have that for some $\iota< \lh(\T_\gg)$, $P_\theta\insegeq \M_\iota^{\T_\gg}$. Because $\rho(P_\theta)=\l$ and $P_\theta$ is sound, we must have that $\P_\theta\insegeq W_\zeta$. Thus, $P_\xi\insegeq W_\zeta$, and therefore $\T_\zeta$ witnesses that IH$_\xi$ holds. 
\end{proof}

We thus have that our construction lasts $\d_y$ many steps. Let then $\T$ be according to $\Lambda$ and such that $\cop(\T)=P_{\d_y}$. The following now follows from our backgrounding condition and from the results of \cite[Chapter 12]{FSIT} (and in particular, from \cite[Lemma 11.4]{FSIT}).\\\\
(A) Suppose $F^*\in \vec{E}^{\N^*_y}$ is such that for some $\N^*_y$-inaccessible cardinal $\l>\k_\a$, $\cp(F^*)>\k_\a$, $\nu(F^*)=\lambda$ and $P_{\d_y}|\lambda=\pi_{F^*}^{\N^*_y}(P_{\d_y})|\lambda$. Let $\tau=\cp(F^*)$ and $F=_{def}F^*\cap P_{\d_y}$. Let $\rho$ be such that $\rho$ is the natural length of $F\rest \rho$ and $\rho$ is not a generator of $F$. Then the natural completion of $F\rest \rho$ is on the sequence of $P_{\d_y}$. \\\\
Let $b=\Lambda(\T)$. Because of clause 1 of the construction, we must have that $\M^\T_b\models ``\d_y$ is a Woodin cardinal" (notice that $b\not \in \N^*_y$ but $\M^\T_b$ is a class of $\N^*_y$). We then set $M_{\a+1}=\M^\T_b$. This finishes our construction of $(M_\a, \k_\a, j_{\a, \b}:\a\leq \omega_1)$.

Set now $\P=j_{0, \omega_1}(\N)|\omega_1$ and let $\Sigma$ be the strategy of $\P$ induced by $(\Sigma')_{M_{\omega_1}}$. Notice that because in $V[g]$, $(\Sigma')_{M_{\omega_1}}$ is an $\omega_2$-strategy, $\Sigma$ is an $\omega_2$-strategy. We have that $(\P, \Sigma)$ is a mouse pair. It is not hard to see that $(\P, \Sigma)$ satisfies clauses 1-4\footnote{We get a club of inaccessibles in $\P$ as each $\k_\a$ is such a point.} of \rdef{generic generator}. Clause 5 is a consequence of the fact that $\Delta=\Delta_{gen}$. We verify clause 6.

Towards showing clause 6, fix a mouse pair $(\Q, \Phi)$. We have that $Code(\Phi)\in \Delta$. We can then find $y_\a$ such that $\a=\b+1$ and $Code(\Phi)$ is Suslin, co-Suslin captured by $\N^*_{y_\a}$. Let $\P_\a=j_{0, \a}(\N)$. We claim that $Code(\Phi)$ is projective in $Code(\Sigma_{\P_\a|\kappa_\a})$. 

To see this, let $(\S_\xi, \S_\xi': \xi<\d_{y_\a})$ be the output of the fully background $L[\vec{E}]$-construction\footnote{Because we want to use Steel's Comparison Theorem from \cite{SteelCom}, this construction uses the background condition of \cite{SteelCom}.} of $\P_\a$ that uses extenders with critical points $\geq \omega_1^{\N^*_{y_\a}}$ (because we left $\kappa$ behind, $\P_\a$ has many total extenders whose critical point is $<\omega_1^{\N^*_{y_\a}}$). It follows from the results of \cite[Chapter 8.1]{SteelCom} that for some $\nu<\d_{y_\a}$, $\S_\nu$ is a $\Phi$-iterate of $\Q$ such that the iteration embedding $j:\Q\rightarrow \S_\nu$ exists and $\Phi_{\S_\nu}$ is the strategy of $\S_\nu$ induced by $\Sigma_{y_\a}$\footnote{This follows from (A). Indeed, the proof is straightforward assuming that the $L[\vec{E}]$-construction was performed in $\N^*_y$ itself. However, the construction is performed inside $\P_{\a}$. Nevertheless, the same proof works as (A) above guarantees that $\P_\a$ has many backgrounded extenders. The reader may wish to consult \cite[Chapter 4.5]{LSA}.}. Because $\kappa_\a$ is a $<\d_{y_\a}$-strong cardinal in $\N^*_{y_\a}$, we get that $\nu<\kappa_\a$\footnote{To see this, fix $\N^*_{y_\a}$-inaccessible cardinals $\zeta<\zeta'<\d_{y_\a}$ such that $\P_\a|\zeta$ is constructed inside $\M_{\a+1}|\zeta+1$, $\M_{\a+1}|\zeta+1$ is constructed in $\N^*_{y_\a}|\zeta'$ and $\S_\nu$ is constructed inside $\P_\a|\zeta$. Let $F\in \vec{E}^{\P}$ be such that $\cp(F)=\k_\a$, $\nu(F)>\zeta'$ and that $\nu(F)$ is a measurable cardinal of $\P_\a$. We can find such an $F$ because our original $\k$ was a strong cardinal in $\N$. Let $F'\in \vec{E}^{M_{\a+1}}$ be the background extender of $F$ and $F''\in N^*_{y_\a}$ be the background extender of $F'$. We now have that $\pi_{F''}(\P_\a)|\nu(F)=\P_\a|\nu(F)$. It then follows that in $Ult(\N^*_{y_\a}, F'')$, $\S_\nu$ is built inside $\pi_{F''}(\P_\a)|\pi_{F''}(\k_\a)$. Hence, (because $\Q\in \N^*_{y_\a}|\k_\a$) $\S_\nu$ is built inside $\P_\a|\k_\a$.},\footnote{Here, we use universality, and we need ${\sf{NLE}}$ to implement it.}. Because we have that the fragment of $(\Sigma')_{M_\a}$ that acts on trees above $\kappa_\b$ is induced by $\Sigma_{y_\a}$, it follows that $\Phi_{\S_\nu}$ is also the strategy induced by $(\Sigma')_{\M_\a}$, and hence, by $\Sigma_{\P|\kappa_\a}$ (recall that $\P_\a|\kappa_\a=\P|\kappa_\a$). It follows that $Code(\Phi)$ is projective in $Code(\Sigma_{\P|\kappa_\a})$.
\end{proof}

\begin{definition}\label{g.g. production} Suppose that $(F, x_0, \Gamma^*, (N, \Psi), A)$ is as in the proof of \rthm{existence of gg}. Then we say that the iteration $(M_\a, j_{\a, \a+1}: \a\leq \omega_1)$ is a \textbf{fully backgrounded pe-gg-producing} iteration of $(\N^*_{x_0}, \d_{x_0}, \Sigma_{x_0})$ if the iteration is defined as in the proof of \rlem{existence of gg} relative to some enumeration $(y_\a: \a<\omega_1)$. We say that the iteration is \textbf{relative to} $\k$ if $\k$ is the strong cardinal chosen from $\N$ as in the proof of \rthm{existence of gg}.

Similarly we define the \textit{fully backgrounded hp-gg-producing} iteration of $(\N^*_{x_0}, \d_{x_0}, \Sigma_{x_0})$. If the type of generic generator is irrelevant then we simply use the expression ``gg-producing iteration".
\end{definition}
The key useful property of the generic generator that we have produced is that its extenders have background certificates. 

\begin{definition}\label{backgrounded g.g.} We say that a generic generator $(\P, \Sigma)$ is backgrounded if there is $$(F, x_0, \Gamma^*, (N, \Psi), A)$$ as in the proof of \rthm{existence of gg} such that $(\P, \Sigma)$ is the generic generator produced via some g.g.-producing iteration of $(N^*_{x_0}, \d_{x_0}, \Sigma_{x_0})$. 
\end{definition}

The following is a special case of \rthm{existence of gg}. 
\begin{notation}\label{strong cardinals} Suppose $M$ is a transitive model of some fragment of set theory. We then set:
\begin{enumerate}
\item $S(M)=\{ \k: M\models ``\k$ is a strong cardinal$"\}$.
\item $W(M)=\{ \d: M\models ``\d$ is a Woodin cardinal$"\}$. 
\item For any set $B\subseteq {\sf{Ord}}$, $cB=\{\a: \sup(\a\cap B)=\a\}$.
\end{enumerate}
\end{notation} 

\begin{definition} We let ${\sf{NWLW}}$ stand for the statement: ``There is no $\omega_1+1$-iterable active mouse $\M$ such that $\M\models ``$there is a Woodin cardinal that is a limit of Woodin cardinals"".
\end{definition}

\begin{corollary}\label{existence of gg below wlw} Suppose $V\models {\sf{AD_\bR}}+V=L(\powerset(\bR))+\neg{\sf{LEC}}+{\sf{NWLW}}$.  Then setting $\Delta=\Delta^V_{gen, pe}$, whenever $g\subseteq Coll(\omega_1, \bR)$ is $V$-generic, there is a pe-generic generator $(\P, \Sigma)\in V[g]$ for $\Delta$ such that $\sup(W(\P)\cup S(\P))<\omega_1$ and all strong cardinals of $\P$ are limits of Woodin cardinals. 
\end{corollary}
\begin{proof} In the proof of \rthm{existence of gg}, $\k$ was chosen to be any strong cardinal of $\N$. As $\d$ is a Woodin cardinal in $L(\N)$, we have that $\sup(W(\N))<\d$. We can now perform the g.g. producing iteration of \rthm{existence of gg} by choosing $\kappa$ to be the least  strong cardinal of $\N$ such that $\sup(W(\N))<\k$. 
\end{proof}

\begin{remark}\label{gap 0}
\normalfont \rthm{existence of gg} can be used to produce a hp-generic generator with essentially the same proof using hod pair constructions instead of $L[\vec{E}]$ constructions. However, while we strongly believe that Corollary \ref{existence of gg below wlw} is true for hp-generic generators, there is currently a gap in the literature which doesn't allow us to prove it for hod pair generic generators. The issue is that if $\N^*_{x_0}$ is as in the proof of Theorem \ref{existence of gg} then it is not known that $\d_{x_0}$ is a Woodin cardinal in the hod pair construction of  $\N^*_{x_0}$ as developed in \cite[Definition 3.1.1]{SteelCom}. The $L[\vec{E}]$ construction introduced in \cite[Definition 3.2]{FarmBC} does have this property as shown in \cite[Theorem 3.6]{FarmBC}. Instead of \cite[Definition 3.1.1]{SteelCom} one could use \cite[Definition 3.2]{FarmBC} to define hod pair constructions but here the missing piece is that \cite{SteelCom} verifies that backgrounded constructions inherit nice strategies for the hod pair construction introduced in \cite[Definition 3.1.1]{SteelCom}. 

It may seem that this is also an issue for pe-generators, but the issue can be resolved easily by quoting published results. Recall that in Corollary \ref{existence of gg below wlw} we assume ${\sf{NWLW}}$. First, it is shown in \cite{SteelCom} and in \cite[Chapter 3, Chapter 4, Claim 4.1]{ESC} that the mouse construction of $\N^*_{x_0}$ is universal in the sense that if $\N$ is the output of this construction then $S(\N)$, the stack over $\N$, has covering in the sense that $\N^*_{x_0}\models \cf(S(\N)\cap {\sf{Ord}})\geq \d_{x_0}$. The results of \cite{FuchsTM} imply that $S(\N)$ can be translated into an ms-mouse, call it $\M$. Let now $\N'$ be the $L[\vec{E}]$-construction of $\N^*_{x_0}$ done as in \cite{FSIT}. Thus, $\N'$ is also an ms-mouse and moreover, $S(\N')\models ``\d_{x_0}$ is a Woodin cardinal" and $\N^*_{x_0}\models \cf(S(\N')\cap {\sf{Ord}})\geq \d_{x_0}$. Let $\M'=S(\N')$. We can now compare $\M$ and $\M'$. Because both are universal, we get that they embed into the same mouse $\Q$ such that if $i: \M\rightarrow \Q$ and $i':\M'\rightarrow \Q$ are the iteration embeddings\footnote{To iterate $\M$ and $\M'$, we use their strategies induced by $\Sigma_{x_0}$.} then $i(\d_{x_0})=i'(\d_{x_0})=\d_{x_0}$. It now follows that $\M\models ``\d_{x_0}$ is a Woodin cardinal" and hence, $L[\N]\models ``\d_{x_0}$ is a Woodin cardinal". 
\end{remark}

\section{Amenable linear iterations}

The main result of this section is \rcor{cor to equality}. It will be used in the proof of \rthm{towards wlw}. The reader may choose to skip this section. Given a transitive set $M$ and an $M$-extender $E$, we say $E$ is \textit{amenable} to $M$ if for any $X\in M$ and $a\in (\lh(E))^{<\omega}$, $E_a\cap X\in M$. Recall that in an iteration tree, each extender is \textit{close} to the model it is applied to (see \cite[Lemma 6.1.5]{FSIT}). It follows that if $M\models {\sf{ZFC-Powerset}}$ and $\T$ is an iteration tree on $M$ and $\pi^\T$ exists then if $N$ is on the main branch of $\T$ and $E$ is applied to $N$ then $E$ is amenable to $N$. 

\begin{definition} Given an lbr pmouse or just a premouse $\R$, we say $p=(\R_\a, E_\a, \pi_{\a, \b}: \a<\b<\iota)$ is an \textit{amenable linear iteration} of $\R$ if 
\begin{enumerate}
\item $\R_0=\R$, and for all $\a<\iota$, $\R_\a$ is transitive,
\item $E_\a$ is a total $\R_\a$-extender that is amenable to $\R_\a$,
\item for $\a<\iota$, $\R_{\a+1}=Ult(\R_\a, E_\a)$ and $\pi_{\a, \a+1}=\pi^{\R_\a}_{E_\a}$,
\item for $\a<\b<\iota$, $\pi_{\a, \b}:\R_\a\rightarrow \R_\b$ is the composition of $\pi^{\R_\xi}_{\xi, \xi+1}$ for $\xi\in [\a, \b)$,
\item for limit $\l<\iota$, $\R_\l$ is the direct limit of $(\R_\a, \pi_{\a, \b}: \a<\b<\l)$ and $\pi_{\a, \l}:\R_\a\rightarrow \R_\l$ is the direct limit embedding.
\end{enumerate}
If $p$ has a last model then we let $\pi_p=\pi_{0, \iota-1}$.  
\end{definition}

Our main source of amenable linear iterations are the main branches of iteration trees. Suppose $\T$ is an iteration tree on $\R$ satisfying ${\sf{ZFC-Powerset}}$. Assume $\pi^\T$ exists and let $\eta$ be an inaccessible cardinal of $\R$. Suppose all extenders used on the main branch of $\T$ have critical points below the image of $\eta$. More precisely, if $\a<lh(\T)$ is on the main branch of $\T$ and $F$ is the extender used at $\a$ on the main branch then $\cp(F)<\pi^\T_{0, \a}(\eta)$. Let now $(\R_\a, E_\a: \a<\iota)$ be the models and extenders used in $\T$ on the main branch of $\T$ (we re-enumerate these models). Set $\M_0=\R|\eta$ and $\M_\a=\pi_{0, \a}^\T(\M_0)$. Then we naturally have an amenable linear iteration $(\M_\a, E_\a, \pi_{\a, \b}:\a<\b<\iota)$ induced by $\T$. 

Suppose now that $p$ is an amenable linear iteration of $\R\models {\sf{ZFC}}$ where $\R$ is a (lbr) premouse. Suppose further that $\M$ is a (lbr) premouse extending $\R$ such that $\eta=_{def}{\sf{Ord}}\cap \R$ is the largest cardinal of $\M$, $\M\models ``\eta$ is an inaccessible cardinal" and $\M\models {\sf{ZFC-Powerset}}$. We can now apply $p$ to $\M$ and obtain an amenable linear iteration $p^\M$ of $\M$. $p^\M$ is the result of successively applying the extenders used in $p$ to $\M$.

Our goal is to prove \rcor{cor to equality}. We start by proving what is really the main step towards \rcor{cor to equality}.

\begin{lemma}\label{equality from ultrapowers} Suppose $\M$ and $\N$ are two lbr premice or just premice having the same largest cardinal $\eta$ such that
\begin{enumerate}
\item $\M, \N\models {\sf{ZFC-Powerset}}$,
\item $\M, \N\models ``\eta$ is an inaccessible cardinal" and
\item $\M|\eta=\N|\eta$.
\end{enumerate}
Suppose $E$ is an extender such that
\begin{enumerate}
\item $E$ is amenable to $\M$ and $\N$,
\item $\cp(E)<\eta$ and
\item $Ult(\M, E)=Ult(\N, E)$ and $Ult(\M, E)$ is well-founded\footnote{The referee has pointed out that the well-foundedness of these ultrapowers is not necessary.}. 
\end{enumerate} Then $\M= \N$.
\end{lemma}
\begin{proof} We assume that $\M$ and $\N$ are premice. We first show that $\N\subseteq \M$. Suppose $\N_0\inseg \N$ is such that $\rho(\N_0)=\eta$. Without losing generality, we suppose $\rho_1(\N_0)=\eta$. Otherwise we could work with master codes. We claim that $\N_0\in \M$. 

To see this, let $f\in \M$ and $a\in [\nu_E]^{<\omega}$ be such that $\pi_E^\M(f)(a)=\pi_E^\N(\N_0)$. Given $s\in [\cp(E)]^{\card{a}}$, let $p_s$ be the first standard parameter of $f(s)\insegeq \M$. Let $p$ be the first standard parameter of $\N_0$. Let $n$ be the length of $p$ and let $\mathcal{L}$ be the language of premice augmented by $n$ constant symbols $c_1,..., c_n$ and constants for each $\a<\eta$. The intended interpretation of $c_1,..., c_n$ is the standard parameter of the structure. 

Given $t\in [\eta]^{<\omega}$ and a formula $\phi(c_1,...,c_n, x_1,..., x_k)$, set 
\begin{center}
$A_{\phi, t}=\{ s\in [\cp(E)]^{\card{a}}: (f(s), p_s)\models \phi[c_1,..., c_n, t]\}$.
\end{center}
We have
\begin{center}
$\pi_E^\N(\N_0)\models \phi[\pi_E^\N(p), \pi_E^\N(t)]$ if and only if $A_{\phi, t}\in E_a$.
\end{center}
Because $E$ is amenable to $\M$, we have that 
\begin{center}
$S=_{def}\{ A_{\phi, t}\in E_a: \phi$ a formula and $t\in [\eta]^{<\omega}\}\in \M$.
\end{center}
Let $T=\{(\phi, t): A_{\phi, t}\in S\}$. $T$ is a consistent complete $\mathcal{L}$-theory implying that there is $\N^*\in \M$ that is the model of $T$. We can assume that $\N^*$ is generated by $c_1^{\N^*},..., c_n^{\N^*}$ and finite sequences from $\eta$. It follows that we have an embedding\footnote{$Hull_1$ is the $\Sigma_1$ Skolem hull.}
\begin{center}
$j: \N^*\rightarrow  Hull_1^{\pi_E^{\N}(\N_0)}(\pi_E^\N(p)\cup \pi_E^\N[\eta])$.
\end{center}
Clearly $j$ is surjective. As $\N_0$ is the transitive collapse of $Hull_1^{\pi_E^{\N}(\N_0)}(\pi_E^\N(p)\cup \pi_E^\N[\eta])$, we have that $\N^*=\N_0$, implying that $\N_0\in \M$. 

We now have that $\N\subseteq \M$. Similarly, $\M\subseteq \N$. It follows that $\lfloor \M \rfloor=\lfloor \N \rfloor$\footnote{$\lfloor K \rfloor$ is the universe of structure $K$.}, implying that $\pi^\M_E=\pi^\N_E$. Therefore, $\M=\N$ as premice.
\end{proof}

\begin{corollary}\label{cor to equality} Suppose $\M$ and $\N$ are two lbr premice or premice having the same largest cardinal $\eta$ such that
\begin{enumerate}
\item $\M, \N\models {\sf{ZFC-Powerset}}"$,
\item $\M, \N\models ``\eta$ is an inaccessible cardinal", and
\item $\M|\eta=\N|\eta$.
\end{enumerate}
Suppose $(\W_\a, E_\a, \pi_{\a, \b}: \a<\b<\iota)$ is an amenable linear iteration of $\M|\eta$ with a last model. Assume that the last models of $p^\M$ and $p^\N$ coincide. Then $\M= \N$.
\end{corollary}
\begin{proof} Let $p^\M=(\M_\a, E_\a, \pi_{\a, \b}^\M: \a<\b<\iota)$ and $p^\N=(\N_\a, E_\a, \pi_{\a, \b}^\N: \a<\b<\iota)$. We again start by showing that $\N\subseteq \M$. Let $\N_0\inseg \N$ be such that $\rho(\N_0)=\eta$. We want to show that $\N_0\in \M$.

Let $\b<\iota$ be the least such that $\pi^\N_{0, \b}(\N_0)\in\M_\b$. We must have that $\b=0$ or $\b=\a+1$. Indeed, if $\b>0$ is a limit ordinal then there is $\a<\b$ such that $\pi^\N_{0, \b}(\N_0)\in \rge(\pi_{\a, \b}^\M)$. It follows that $\pi_{0, \a}^\N(\N_0)\in \M_\a$. Suppose then $\b=\a+1$. It follows from \rlem{equality from ultrapowers} that $\pi_{0, \a}^\N(\N_0)\in \M_\a$. Therefore, $\b=0$.

We now have that $\N\subseteq \M$. The same proof shows that $\M\subseteq \N$, implying that $\lfloor \M \rfloor=\lfloor \N \rfloor$ and hence, $\pi_{p^\M}=\pi_{p^\N}$. Therefore, $\M=\N$.
\end{proof}

\section{Computing the powerset operation: Towards a Woodin cardinal that is a limit of Woodin cardinals}

The ${\sf{Direct\ Limit\ Independence\ Hypothesis}}$  essentially says that $\M_\infty(\P, \Sigma)$ is independent of $(\P, \Sigma)$. Below we give a more precise definition.

Assume ${\sf{ZFC}}$, and suppose $(\P,\Sigma)$ is a pure extender pair such that 
\begin{enumerate}
\item ${\sf{Ord}}\cap \P=\omega_1$,
\item $\Sigma$ is an $\omega_2$-strategy, and
\item $\Sigma$ acts on $L[\P]$.
\end{enumerate} 
We let $\mathcal{F}(\P, \Sigma)=\{ \R: \R$ is a $\Sigma$-iterate of $\P$ via a countable tree\footnote{This means that $\lh(\T)<\omega_1$.} $\T$ such that $\pi^\T$ exists$\}$. We let $\M_\infty(\P, \Sigma)$ be the direct limit of $\mathcal{F}(\P, \Sigma)$. 

\begin{definition}[${\sf{Direct\ Limit\ Independence\ Hypothesis}}$]\label{dli} Assume $V\models {\sf{AD_{\mathbb{R}}}}+{\sf{NLE}}+V=L(\powerset(\bR))$. Suppose $\Delta\subseteq \powerset(\bR)$ and $g\subseteq Coll(\omega_1, \bR)$ is $V$-generic. We then say that ${\sf{DLIH}}$ holds for $\Delta$ in $V[g]$ if whenever $(\P, \Sigma), (\Q, \Lambda)\in V[g]$ are two generic generators for $\Delta$, either 
\begin{center}
$\M_\infty(\P, \Sigma)\insegeq \M_\infty(\Q, \Lambda)$ or $\M_\infty(\Q, \Lambda)\insegeq \M_\infty(\P, \Sigma)$.
\end{center}
We write $V[g]\models {\sf{DLIH}}(\Delta)$ to mean that ${\sf{DLIH}}$ holds for $\Delta$ in $V[g]$. 
\end{definition}

In this section, we show that the ${\sf{Direct\ Limit\ Independence\ Hypothesis}}$ (${\sf{DLIH}}$) can be used to compute the powerset operation in generic generators. We will use this result to show that ${\sf{DLIH}}$ implies that if there is a non-generated set of reals then the $L[\vec{E}]$ construction of an $\N^*_x$-model reaches a Woodin cardinal that is a limit of Woodin cardinals. 

\begin{remark}\label{gap 1}
\normalfont
Our results also carry over to hod pair generic generators: meaning assuming one can resolve the issue raised in \rrem{gap 0}, our results will show that if ${\sf{HPC}}$ fails then the hod pair construction of some $\N^*_x$-model reaches a Woodin cardinal that is a limit of Woodin cardinals.
\end{remark}

\begin{terminology}
\normalfont Suppose $(\P, \Sigma)$ is a mouse pair. Let $\nu_\P=\sup(W(\P))$ and $\P^b=\P|(\nu^+_\P)^\P$. We say $\eta$ is a \textbf{weak cutpoint} of $\P$ if $\eta>\nu_\P$ and all extenders overlapping $\eta$ have critical points $\leq\nu_\P$.
\end{terminology}

\rlem{comp of powerset} provides a definition of $\P|(\eta^+)^\P$ from $\P|\eta$, where $\eta$ is a weak cutpoint of $\P$. The aforementioned computation shows that $\P|(\eta^+)^\P$ is simple, and yields a contradiction when we also assume that $(\P, \Sigma)$ is backgrounded.


\begin{theorem}\label{comp of powerset} Assume $V\models {\sf{AD_{\mathbb{R}}}}+{\sf{NLE}}+V=L(\powerset(\bR))$. Suppose $(\Delta, (\P, \Sigma), g)$ is such that
\begin{enumerate}
\item $\Delta\subseteq \powerset(\bR)$ and $\Delta$ is closed under Wadge reducibility,
\item $g\subseteq Coll(\omega_1, \bR)$ is $V$-generic,
\item $(\P, \Sigma)\in V[g]$ is a pe-generic generator for $\Delta$ and
\item $V[g]\models {\sf{DLIH}}(\Delta)$.
\end{enumerate}
 Set $X=\pi^\Sigma_{\P, \infty}[\P^b]$, and let $\eta$ be a weak cutpoint of $\P$ that is an inaccessible cardinal of $\P$. Then $\P|(\eta^+)^\P\subseteq OD^V_{\P|\eta, X}$. 
\end{theorem}
\begin{proof} We can assume that $\eta$ is not a measurable cardinal of $\P$. If it is then we can change $\P$ with $Ult(\P, E)$ where $E\in \vec{E}^\P$ is the Mitchell order $0$ extender such that $\cp(E)=\eta$. We then have that $Ult(\P, E)|(\eta^+)^\P=\P|(\eta^+)^\P$, and (since $\eta>{\sf{Ord}}\cap \P^b$) setting $\P_E=Ult(\P, E)$, $X=\pi^{\Sigma_{\P_E}}_{\P_E, \infty}[\P^b]$ (notice that $\P^b=\P_E^b$).

Set $\N=\M_\infty(\P, \Sigma)$. Because $Coll(\omega_1, \bR)$ is a homogenous poset, ${\sf{DLIH}}(\Delta)$ implies that $\N\in \H^V$. We claim that $\P|\eta\insegeq\W\insegeq \P|(\eta^+)^\P$ if and only if it is forced by $Coll(\omega_1, \bR)$ that there is a generic generator $(\Q, \Lambda)$ for $\Delta$ such that 
\begin{enumerate}
\item $\Q^b=\P^b$, 
\item $\Q|\eta=\P|\eta$,
\item $\eta$ is a weak cutpoint of $\Q$, 
\item $\eta$ is a non-measurable inaccessible cardinal of $\Q$, 
\item $\pi^{\Lambda}_{\Q, \infty}[\Q^b]=X$,
\item $\M_\infty(\Q, \Lambda)=\N$ and
\item $\W\inseg \Q$, 
\end{enumerate}
To see the right-to-left direction, fix $q\in Coll(\omega_1, \bR)$. As $V[g]=V[g_q]$, we have that in $V[g_q]$ we can take $(\Q, \Lambda)=(\P, \Sigma)$. 

Fix then a condition $q\in Coll(\omega_1, \bR)$ and $(\dot{\Q}, \dot{\Lambda})\in V^{Coll(\omega_1, \bR)}$ such that $q\forces ``(\dot{\Q}, \dot{\Lambda})$ has properties 1-7". We want to see that $\W\insegeq \P$.  It is enough to work in $V[g]$ and show that if $(\Q, \Lambda)=(\dot{\Q}(g_q), \dot{\Lambda}(g_q))$ then $\P|(\eta^+)^\P=\Q|(\eta^+)^\Q$. 

Because ${\sf{AD_\bR}}$ holds, we can find a set of reals $B$ such that $\{(\Q, \Lambda), (\P, \Sigma)\}\in L(B, \bR)[g]$ and $(B, \bR)^\#$ exists. Fix then a countable $Z\prec (B, \bR)^\#$ such that letting $\pi: (B_\sigma, \sigma)^\#\rightarrow Z$ be the transitive collapse and $N=_{def}L(B_\sigma, \sigma)$,
\begin{enumerate}
\item $g_\sigma=_{def}g\cap Coll(\omega_1^N, \sigma)$ is $N$-generic,
\item $\pi$ extends to $\pi^+:N[g_\sigma]\rightarrow L(B, \bR)[g]$,
\item $\{(\P, \Sigma), (\Q, \Lambda)\}\subseteq \rge(\pi^+)$, 
\item $\eta<\omega_1^N$ and
\item $\omega_1^N$ is inaccessible in both $\P$ and $\Q$ (to ensure this condition we use clause 3 of Definition \ref{generic generator}).
\end{enumerate}
Let $\M$ be such that $\pi^+(\M)=\N$, and set $\P'=\P|\omega_1^N$ and $\Q'=\Q|\omega_1^N$. We have that $\M$ is a $\Sigma_{\P'}$-iterate of $\P'$ and a $\Lambda_{\Q'}$-iterate of $\Q'$. Let $\T'$ be $\P'$-to-$\M$ tree according to $\Sigma_{\P'}$ and $\U'$ be the $\Q'$-to-$\M$ tree according to $\Lambda_{\Q'}$. Because $\omega_1^N$ is inaccessible in both $\P$ and $\Q$, letting $\T$ and $\U$ be the copies of $\T'$ and $\U'$ onto $\P$ and $\Q$ respectively, and letting $\P^+$ and $\Q^+$ be the last models of $\T$ and $\U$ respectively, we have that $\P^+=\P|\xi$ and $\Q^+=\Q|\xi$ where $\xi={\sf{Ord}}\cap \M$. 

Set $\nu=\sup(W(\M))$. Let $\R$ be the least node of $\T$ such that $\R|\nu=\M|\nu$ and let $\S$ be the least node of $\U$ such that $\S|\nu=\M|\nu$. Notice that because $\P^+$ and $\Q^+$ have no Woodin cardinals above $\nu$, $\T_{\geq \R}$ and $\U_{\geq \S}$\footnote{These are the portions of $\T$ and $\U$ that come after $\R$ and $\S$.} are two normal trees on $\R$ and $\S$ respectively that are above $\nu$. 

We now have that\\\\
(1) $\pi^\Sigma_{\P, \infty}=\pi^{\Sigma_\R}_{\R, \infty}\circ \pi^{\Sigma}_{\P, \R}=\pi^{\Sigma_{\P^+}}_{\P^+, \infty}\circ \pi^{\Sigma_\R}_{\R, \P^+}\circ \pi^{\Sigma}_{\P, \R}$,\\
(2) $\pi^\Lambda_{\Q, \infty}=\pi^{\Lambda_\S}_{\S, \infty}\circ \pi^{\Lambda}_{\Q, \R}=\pi^{\Lambda_{\Q^+}}_{\Q^+, \infty}\circ \pi^{\Lambda_\S}_{\S, \Q^+}\circ \pi^{\Lambda}_{\Q, \S}$,\\
(3)  $\pi^{\Sigma}_{\P, \infty}[\P^b]=\pi^{\Lambda}_{\Q, \infty}[\Q^b]$ and $\P^b=\Q^b$,\\
(4) $\pi^{\Sigma_{\R}}_{\R, \P^+}[\M^b]=\pi^{\Lambda_{\S}}_{\S, \Q^+}[\M^b]=id$,\\
(5) $\pi^{\Sigma_{\P^+}}_{\P^+, \infty}[\M^b]=\pi[\M^b]=\pi^{\Sigma_{\Q^+}}_{\Q^+, \infty}[\M^b]$.\\\\
It follows that\\\\
(6) $\pi^{\Sigma}_{\P, \R}[\P^b]=\pi^{\Lambda}_{\Q, \S}[\Q^b]$.\\\\
Because the generators of $\T_{\leq \R}$ and $\U_{\leq \S}$ are contained in $\nu$, it follows from (6) that\\\\
(7) $\pi^{\Sigma}_{\P, \R}[\P|\eta]=\pi^{\Lambda}_{\Q, \S}[\Q|\eta]$.\\\\
Since $\eta$ is a non-measurable inaccessible cardinal both in $\P$ and $\Q$, it follows from (7) that\\\\
(8) $\pi^{\Sigma}_{\P, \R}(\eta)=\pi^{\Lambda}_{\Q, \S}(\eta)$.\\\\
Set $\xi=\pi^{\Sigma}_{\P, \R}(\eta)$ and $\V=\R|\xi=\S|\xi$. Notice now that both $\R$-to-$\M$ and $\S$-to-$\M$ iterations proceed by first iterating $\V$ to an initial segment of $\M$. Because $\M$ has no Woodin cardinals above $\nu$, the aforementioned iterations produce the same tree on both $\R$ and $\S$. We let $\K$ be this tree. It is based on $\V$. 

We must now have that $\Sigma_{\R}(\K)=\Lambda_\S(\K)$. Towards a contradiction assume that $\Sigma_\R(\K)\not =\Lambda_\S(\K)$. Let $b=\Sigma_\R(\K)$ and $c=\Lambda_\S(\K)$. Notice that since $\eta$ is a weak cutpoint, both $\Q(b, \K)$ and $\Q(c, \K)$ must exist and moreover, $\Q(b, \K)$ and $\Q(c, \K)\insegeq \M$. It follows that $\Q(b, \K)=\Q(c, \K)$, implying that $b=c$. 

To finish, we will use \rcor{cor to equality}. Set $\P_0=\P|(\eta^+)^\P$ and $\Q_0=\Q|(\eta^+)^\Q$. We first claim that $\pi^{\Sigma}_{\P, \R}(\P_0)=\pi^{\Lambda}_{\Q, \S}(\Q_0)$. To see this, let $\P_1=\pi^{\Sigma}_{\P, \R}(\P_0)$ and $\Q_1=\pi^{\Lambda}_{\Q, \S}(\Q_0)$. Let $d$ be the branch of $\K$ according to both $\Sigma_\R$ and $\Lambda_\S$. We then have that $\pi^\K_d(\P_1)=\pi^\K_d(\Q_1)$. It follows from \rcor{cor to equality} that $\P_1=\Q_1$. 

To conclude that $\P_0=\Q_0$, notice that because $\pi^{\Sigma}_{\P, \R}[\P^b]=\pi^{\Lambda}_{\Q, \S}[\Q^b]$ and the generators of both $\T_{\leq \R}$ and $\U_{\leq \S}$ are contained in $\M|\nu$, the extenders used on the main branch of $\P$-to-$\R$ coincide with the extenders used on the main branch of $\Q$-to-$\S$. It follows that there is an amenable linear iteration $p$ of $\P|\eta$ such that $p^{\P_0}$ and $p^{\Q_0}$ have the same last model. It follows that $\P_0=\Q_0$ (see \rcor{cor to equality}).
\end{proof}

\subsection{A Woodin cardinal that is a limit of Woodin cardinals}

In this subsection, we use our computation of powersets of generic generators to show that assuming ${\sf{DLIH}}$, the boundedness of Woodin cardinals of direct limits of generic generators and the existence of a non-generated set implies that there is a mouse with a Woodin cardinal that is a limit of Woodin cardinals in a somewhat stronger sense.

\begin{definition}\label{boundedness of bottom part} 
    Assume $V\models {\sf{AD_{\mathbb{R}}}}+{\sf{NLE}}+V=L(\powerset(\bR))$. Suppose $\Delta\subseteq \powerset(\bR)$ and $g\subseteq Coll(\omega_1, \bR)$ is $V$-generic. We then say that $\Delta$ has a \textbf{bounded bottom part} if  there is $\gg<\Theta^V$ such that whenever $(\P, \Sigma)\in V[g]$ is a generic generator for $\Delta$, $$\sup({\sf{Ord}}\cap \pi_{\P, \infty}^\Sigma(\P^b))<\gg.$$ 
    We then say that $\gg$ bounds the bottom part of $\Delta$.
\end{definition}

The proof of the next theorem uses standard arguments from descriptive inner model theory. These arguments are, by now, folkloric, but many of them originate in the work of John Steel and Hugh Woodin. Particularly important for us is Woodin's proof of the Mouse Set Conjecture in $L(\bR)$. The reader may wish to consider \cite[Chapter 4]{LSA} for notions such as Suslin capturing and other similar concepts. This material is also covered in \cite[Chapter 7.2]{SteelCom}. See also \cite[Theorem 10.3]{DMATM}, \cite[Theorem 11.3]{DMATM} and \cite[Lemma 2.4]{MSCR} (the last two results will use to establish (2) in the proof of \rthm{towards wlw}). 

\begin{theorem}\label{towards wlw} Suppose $V\models ``{\sf{AD_{\mathbb{R}}}}+{\sf{NLE}}+V=L(\powerset(\bR))+\neg {\sf{LEC}}$. Set $\Delta=\Delta_{gen, pe}^V$ and suppose that $Coll(\omega_1, \bR)$ forces that ${\sf{DLIH}}(\Delta)$ holds and that $\Delta$ has a bounded bottom part. Fix $\gg<\Theta^V$ that bounds the bottom part of $\Delta$, and suppose $\Gamma_0$ is a good pointclass such that $\Delta\subseteq \utilde{\Delta}_{\Gamma_0}$, and if $\gg'$ is the least member of the Solovay sequence such that $\gg<\gg'$ then $\gg'<\d_{\Gamma_0}$\footnote{This condition implies that there is a pre-well-ordering $\leq^*\in \Gamma$ whose length is at least $\gamma$.}. Let $\Gamma$ be a good pointclass such that $\Gamma_0\subseteq \utilde{\Delta}_{\Gamma}$ and let $(F, x_0, \Gamma^*, (N, \Psi), A)$ be as in the proof of \rthm{existence of gg} relative $\Gamma$. Then for any $x$ Turing above $x_0$, letting $F(x)=(\N^*_x, \M_x, \d_x, \Sigma_x)$, $\d_x$ is a limit of Woodin cardinals in the $L[\vec{E}]$-construction of $\N^*_x|\d_x$.
\end{theorem}
\begin{proof}
Let $g\subseteq Coll(\omega_1, \bR)$ be generic and fix $x\in \bR$ which is Turing above $x_0$. Set $M=M_x$, $\d=\d_x$ and suppose that $\d$ is not a limit of Woodin cardinals of the $L[\vec{E}]$-construction of $M|\d$. Let $\P^*$ be the output of the $L[\vec{E}]$-construction of $M|\d$ and let $\k$ be the least strong cardinal of $M|\d$ that is bigger than $\nu=_{def}\sup(W(\P^*))$. Working in $V[g]$, let $\vec{M}=(M_\a, j_{\a, \b}: \a<\b<\omega_1)$ be a pe-gg-producing iteration of $M$ relative to $\k$. Let $\P=j_{0, \omega_1}(\P^*)$ and $\k_\a=j_{0, \a}(\k)$. Recall that $\cp(j_{\a, \a+1})=\k_\a$. Let $\Sigma$ be the strategy of $\P$ induced by $(\Sigma_x)_{M_{\omega_1}}$.

Let $B\in \utilde{\Delta}_{\Gamma^*}$ be a set of reals coding a self-justifying system $(D_i: i<\omega)$ such that $D_0=\{ (x, y)\in \bR^2: x\in C_{\Gamma}(y)\}$.

Let $X=\pi^{\Sigma}_{\P, \infty}[\P^b]$. It follows from \rlem{comp of powerset} that for each inaccessible weak cutpoint $\eta$ of $\P$, $\P|(\eta^+)^\P\in OD^V_{\P|\eta, X}$. Because $\gg$ bounds the bottom part of $\Delta$, we have that $X\cap {\sf{Ord}}\subseteq \gg$. We use this fact to establish (1.1) below. It follows from our choice of $\Gamma$ that (see the paragraph before the statement of \rthm{towards wlw}).\\\\
(1) for a cone of $y\in \bR$, \\
 (1.1) $OD^V_X(y)\cap \bR\subseteq C_{\Gamma_0}(y)$\\ 
 (1.2) $C_{\Gamma_0}(y)\subseteq C_{\Gamma}(y)$\footnote{$C_{\Gamma_0}$ is the largest countable $\Gamma$ set. See \cite{Ma08}.}, and\\
 (1.3) $(\N^*_y, \d_y, \Sigma_y)$ Suslin, co-Suslin captures $B$.\\\\
 We let $z$ be a base of the cone that satisfies (1). Let $(y_\a: \a<\omega_1)$ be the generic enumeration used to produce $\vec{M}$. We have that for some $\a<\omega_1$, $z\in \N^*_{y_\a}$. Fix such an $\a$ and let $Q=\N^*_{y_{\a+1}}$ and $\d_{y_{\a+1}}=\xi$. 

We have that $M_{\a+1}$ is obtained via a variant of the usual backgrounded construction done inside $Q$ over $M_\a|\kappa_\a$ (see \rthm{existence of gg}). We also have that there is a club $U\in Q\cap \powerset(\xi)$ such that\\\\
(2) for every $\b\in U$, $C_{\Gamma}(Q|\beta)\models ``\beta$ is a Woodin cardinal".\footnote{The reader may wish to consult \cite[Theorem 11.3]{DMATM} and \cite[Lemma 2.4]{HMMSC}. The proof of (2) follows these proofs. Essentially, $\Q$ term captures a sjs for $\Gamma$, and any transitive Skolem hull containing these terms witnesses (2). See the discussion before \cite[Theorem 1.9]{ADRUB} and \cite[Chapter 2.4]{HMMSC}} \\\\
We fix now a club $E\in Q\cap \powerset(\xi)$ such that\\\\
(3) for all $\b\in E$, $M_{\a+1}|\b$ and $j_{0, \a+1}(\P^*)|\b$ are produced via the relevant constructions of $Q|\b$ and $M_{\a+1}|\b$ respectively.\\\\\
Fix then $\b<\kappa_{\a+1}$ such that\\\\
(4) $M_{\a+1}|\b$ and $j_{0, \a+1}(\P^*)|\b$ are produced via the relevant constructions of $Q|\b$ and $M_{\a+1}|\b$ respectively, and $C_{\Gamma}(Q|\beta)\models ``\beta$ is a Woodin cardinal".\\\\
Such a $\b$ exists because $\k_{\a+1}$ is a $<\xi$-strong cardinal of $Q$ and $j_{0, \a+1}(\P^*)$\footnote{See the proof outlined in the footnote appearing in the last paragraph of the proof of \rthm{existence of gg}.}. Because $$\P|(\b^+)^\P=j_{0, \a+1}(\P^*)|(\b^+)^{j_{0, \a+1}(\P^*)}$$ and $\b>{\sf{Ord}}\cap \P^b$, it follows from (1.1)-(1.4), (4) and Theorem \ref{comp of powerset} that\\\\
(5) there is continuous chain $(X_\a: \a<\b)\in C_{\Gamma}(Q|\b)$ such that\\
(5.1) for every $\a<\b$, $C_\Gamma(Q|\b)\models \card{X_\a}<\b$,\\
(5.2) for every $\a<\b$, $X_\a\prec C_{\Gamma_0}(Q|\b)$,\\
(5.3) there is a club $K\in \powerset(\b)\cap C_{\Gamma}(Q|\b)$ such that for every $\a\in K$, if $W_\a$ is the transitive collapse of $X_\a$, then letting $\nu_\a=X_\a\cap \b$, $\powerset(\nu_\a)\cap \P\subseteq \powerset(\nu_\a)\cap W_\a$.\footnote{To get the club $K$ and the containment $\powerset(\nu_\a)\cap \P\subseteq \powerset(\nu_\a)\cap W_\a$, we use (1.1)-(1.4). Because of our choice of $\Gamma_0$ and $\Gamma$, $Q$ captures a sjs for $\Gamma_0$, implying that for every set $a\in W_\a$, $C_{\Gamma_0}(W_\a|\nu_\a, a)\subseteq W_\a$. This is the standard condensation lemma for sjs, see \cite[Lemma 2.19]{HMMSC}. We get that $\powerset(\nu_\a)\cap \P\subseteq \powerset(\nu_\a)\cap W_\a$ because (1.1) implies that $\powerset(\nu_\a)\cap \P\subseteq C_{\Gamma_0}(W_\a|\nu_\a, \P|\nu_\a)$.}\\\\
It now follows from the results of \cite{ESC} (e.g. see the proof of Claim 4.1 of \cite{ESC}) that there is a superstrong extender on the sequence of $\P|\b$ whose critical point is $>\nu$.\footnote{Notice that we built $\P$ using the backgrounding condition of \cite[Definition 3.1]{SteelCom}, which is the same as the backgrounding condition of \cite{ESC}. This is where we use the fact that $\b$ is a Woodin cardinal in $C_\Gamma(Q|\b)$: the universality argument presented in \cite[Claim 4.1]{ESC} and\cite[Chapter 2.3]{HMMSC} uses a Woodin cardinal.} Hence, $\nu<\sup(W(\P))$. 
\end{proof}

\section{Computing the powerset operation: towards superstrong cardinals}

In this section, we show how ${\sf{DLIH}}$ and the ${\sf{Bounded\ Direct\ Limits}}$ (${\sf{BDL}}$) can be used to prove ${\sf{LEC}}$ assuming ${\sf{NLE}}$. We start by introducing (${\sf{BDL}}$). 

\begin{definition}[${\sf{Bounded\ Direct\ Limits}}$]\label{bdl} Suppose $V\models {\sf{AD_\mathbb{R}}}+{\sf{NLE}}+V=L(\powerset(\bR))$ and suppose $\Delta\subsetneq \powerset(\bR)$. Let $g\subseteq Coll(\omega_1, \bR)$ be generic. We say that $\Delta$ has ${\sf{Bounded\ Direct\ Limits}}$ and write $V[g]\models {\sf{BDL}}(\Delta)$ if there is $\gamma<\Theta^V$ such that whenever $(\P, \Sigma)\in V[g]$ is a generic generator for $\Delta$, ${\sf{Ord}}\cap \M_\infty(\P, \Sigma)\subseteq \gg$.
\end{definition}

\rthm{getting superstrongs} is the main theorem of this section. Before we state it, we collect some concepts. 

\begin{definition}\label{background wlw}
Suppose $A\subseteq \bR$.  We say that there is an \textbf{$A$-wlw background}  if there is an externally iterable background $\mathbb{M}=(M, \d, \vec{G}, \Sigma)$ in the sense of \cite[Definition 4.1.4]{LSA} such that 
\begin{enumerate}
\item $M\models ``\d$ is a limit of Woodin cardinals" and
\item $\mathbb{M}$ Suslin, co-Suslin captures the first order theory of $(HC, A, \in)$\footnote{More precisely, letting $T=\{(\phi, \vec{x}): \phi$ is a formula, $\vec{x}\in \bR^{<\omega}$ and $(HC, A, \in)\models \phi[\vec{x}]\}$, $T$ is Suslin, co-Suslin captured by $\mathbb{M}$.}.
\end{enumerate}
\end{definition}
A typical example of an $A$-wlw background is an lbr pair $(M, \Sigma)$ such that for some $\d$, $M\models ``\d$ is a Woodin cardinal that is a limit of Woodin cardinals" and $(M, \d, \Sigma)$ Suslin, co-Suslin captures the first order theory of $(HC, A, \in)$. 

\begin{theorem}\label{getting superstrongs} Assume $V\models {\sf{AD_{\bR}}}+{\sf{NLE}}+V=L(\powerset(\bR))+\neg{\sf{LEC}}$. Assume further that for every $A\subseteq \bR$, there is an $A$-wlw background. Let $\Delta=\Delta_{gen, pe}$ and $g\subseteq Coll(\omega_1, \bR)$ is $V$-generic. Then either $V\models \neg {\sf{DLIH}}(\Delta)$ or $V\models \neg {\sf{BDL}}(\Delta)$
\end{theorem}
\begin{proof} Towards a contradiction we assume that $V\models  {\sf{DLIH}}(\Delta) \wedge {\sf{BDL}}(\Delta)$. Let $\gg<\Theta$ be such that whenever $(\P, \Sigma)\in V[g]$ is a pe-generic generator for $\Delta$, ${\sf{Ord}}\cap \M_\infty(\P, \Sigma)\subseteq \gg$. Let then $A\subseteq \bR$ be such that $\gg<\Theta^{L(A, \bR)}$, $\Delta\subseteq \utilde{\Sigma^2_1(A)}^{L(A, \bR)}$ and $A$ is Wadge above all sets in $\Delta$. It follows from ${\sf{AD_{\bR}}}$ that $(A, \bR)^\#$ exists.

Working in $V[g]$, let $\mH$ be the union of all $\M_\infty(\P, \Sigma)$ where $(\P, \Sigma)$ is a pe-generic generator for $\Delta$. Note that it follows from ${\sf{DLIH}}(\Delta)$ that $\mH\in V$, and it follows from Coding Lemma, ${\sf{BDL}}(\Delta)$ and our choice of $A$ that $\mH\in L(A, \bR)$.

 Given a countable $\sigma\subseteq \bR$, we say $\sigma$ is \textit{almost} $A$-\textit{full} if there is an elementary $j: (A\cap \sigma, \sigma)^\#\rightarrow (A, \bR)^\#$, and we say $\sigma$ is $A$-\textit{full} if there is an elementary $j: (A\cap \sigma, \sigma)^\#\rightarrow (A, \bR)^\#$ such that $(\gg, \Delta, \mH)\in \rge(j)$. It follows that the set of almost $A$-full $\sigma\subseteq \bR$ contains a club in $V[g]$ (we give an explonation below)\footnote{In $V$, the set of such $\sigma$ has measure one with respect to the supercompactness measure on $\powerset_{\omega_1}(\bR)$.}. For an almost $A$-full $\sigma$, set $R^+_\sigma=L(A\cap\sigma, \sigma)$ and $R'_\sigma=L_{\Theta^{R^+_\sigma}}(A\cap \sigma, \sigma)$.

 If $\sigma$ is almost $A$-full then there is a unique embedding  $j: R'_\sigma\rightarrow L_{\Theta^{L(A, \bR)}}(A, \bR)$ with the property that $j(A\cap \sigma)=A$ and $j$ extends to an embedding $j^+: R^+_\sigma\rightarrow L(A, \bR)$. We let $j_\sigma:R'_\sigma\rightarrow L_{\Theta^{L(A, \bR)}}(A, \bR)$ be this unique embedding. Let $x_0\in \bR$ be such that \\\\
(a) there is a $\Delta^1_1(A, x_0)$ set of reals whose Wadge rank is $w(\Delta)=\sup\{w(Z): Z\in \Delta\}$, and\\
(b) there is a $\Delta^1_1(A, x_0)$ set of reals coding $\mH$.\\\\ 
Given an almost $A$-full $\sigma$ with $x_0\in \sigma$, we have that $\{\gg, \Delta, \mH\}\subseteq \rge(j_\sigma)$. Thus, any almost $A$-full $\sigma$ with $x_0\in \sigma$ is $A$-full, and therefore, the set of $A$-full $\sigma$ contains a club in $V[g]$.

Given an $A$-full $\sigma$, we set $\gg_\sigma=j_{\sigma}^{-1}(\gg)$, $\mH_\sigma=j_\sigma^{-1}(\mH)$ and $\Delta_\sigma=j_{\sigma}^{-1}(\Delta)$. Given $A$-full $\sigma\subseteq \tau$, we let $j_{\sigma, \tau}:R'_\sigma\rightarrow R'_\tau$ be the embedding $j_{\tau}^{-1}\circ j_\sigma$, and also we let $\beta_\sigma$ be the least $\b$ such that 
\begin{itemize}
\item $\gg_\sigma<\Theta^{L_\b(A\cap \sigma, \sigma)}$, $\Delta_\sigma\subseteq (\utilde{\Sigma^2_1(A\cap \sigma)})^{L_\b(A\cap \sigma, \sigma)}$ and $\mH_\sigma\in L_\b(A\cap \sigma, \sigma)$,
\item there is a subset  of $\sigma$ definable over $L_{\b}(A\cap \sigma, \sigma)$.
\end{itemize}
Set $R_\sigma=L_{\b_\sigma}(A\cap \sigma, \sigma)$ and let $k_\sigma=j_\sigma\rest R_\sigma$. If $\sigma\subseteq \tau$ are $A$-full then set $k_{\sigma, \tau}=j_{\sigma, \tau}\rest R_\sigma$. Notice that if $\sigma\subseteq \tau$ are $A$-full and $\sigma\in R'_\tau$ then \begin{itemize}
    \item $k_{\sigma, \tau}\in R'_\tau$ and
    \item $k_{\sigma, \tau}(\mH_\sigma, \gg_\sigma, \Delta_\sigma)=(\mH_\tau, \gg_\tau, \Delta_\tau)$.  
\end{itemize} 
We thus have the following data\footnote{The embedding $j^+_{\sigma, \tau}$ is not unique, but the others are. It is not really important how we pick $j^+_{\sigma, \tau}$, but we can just as well let $j^+_{\sigma, \tau}$ be the minimal embedding, i.e. the one that is determined by $j^+(A\cap \sigma)=A\cap \tau$, $j^+\rest \sigma=id$ and $\rge(j^+)$ contains every $A\cap \tau$-indiscernible.}.
\begin{itemize}
\item $R^+_\sigma=L(A\cap \sigma, \sigma)$, $R'_\sigma=L_{\Theta^{R^+_\sigma}}(A\cap \sigma, \sigma)$ and $R_\sigma=L_{\b_\sigma}(A\cap \sigma, \sigma)$,
\item $j^+_{\sigma, \tau}: R_\sigma^+\rightarrow R_\tau^+$, $j_{\sigma, \tau}=j^+_{\sigma, \tau}\rest R'_{\sigma}$ and $k_{\sigma, \tau}=j_{\sigma, \tau}\rest R_\sigma$.
\end{itemize}
Let $B$ be the set of reals that codes the set $\{\sigma\in \powerset_{\omega_1}(\bR): \sigma$ is $A$-full and $x_0\in \sigma\}$, and let $C$ be the set of reals that codes the set $\{(\sigma, \tau)\in B^2: \sigma\subseteq \tau$ and there is $x\in \tau\cap B$ that codes $\sigma\}$. 

Next we let $\Gamma$ be a good pointclass such that $(A, B, C, x_0, \bR)^\#\in \utilde{\Delta}_{\Gamma}$ and let $U\subseteq \omega\times \bR\times \bR$ be a universal $\Gamma$ set. Also, let $(e, x_1)\in \omega\times \bR$ be such that $U_{e, x_1}=(A, B, C, x_0, \bR)^\#$. Finally let $(P, \d, \vec{G}, \Sigma)$ be a $(U, x_1)$-wlw background. Thus, $(x_0, x_1)\in P$. 

Let $\K$ be the $L[\vec{E}]$-construction of $V_\d^P$ and $\Lambda$ be the strategy of $\K$ induced by $\Sigma$. Because $\Sigma$ is term captured over $P$, we have that $\Lambda$ is also term captured over $P$. Our goal is to show that if $\kappa$ is a strong cardinal of $\K$ then $\powerset(\k)\cap \K$ is contained in $C_\Gamma(P|\k)$. This then will imply, as in \cite[Chapter 3, Chapter 4, Claim 4.1]{ESC}, that $\K$ has a superstrong extender. We will elaborate on this point at the of the current proof.

Suppose $\kappa$ is an inaccessible cardinal of $V_\d^P$. Given a $P$-generic $h\subseteq Coll(\omega, <\k)$ let $\sigma_h=\bR^{P[h]}$. 
\begin{claim}\label{claim 1} Suppose $\k<\d$ is an inaccessible cardinal of $P$. Then  for any $P$-generic $h\subseteq Coll(\omega, <\k)$, $\sigma_h$ is $A$-full.
\end{claim}
\begin{proof} Because $P[h]$ Suslin, co-Suslin captures the theory of $(HC, B, \in)$, we have that $(HC^{P[h]}, B\cap P[h], \in)\prec (HC, B, \in)$. Therefore, for any $x\in P[h]\cap \bR$ there is $y\in B\cap P[h]$ such that if $\sigma$ is the set coded by $y$ then $x\in \sigma$. Thus, we can find $(\sigma_\a: \a<\omega_1^{P[h]})\in P[h]$ such that 
\begin{itemize}
    \item $\sigma_\a$ is $A$-full, 
    \item for $\a<\b<\omega_1^{P[h]}$, $\sigma_\a$ is coded by some real in $B\cap \sigma_\b$,
    \item $\bR^{P[h]}=\cup_{\a<\omega_1^{P[h]}}\sigma_\a$.
\end{itemize}
It then follows that $L(A\cap \bR^{P[h]}, \bR^{P[h]})$ is the direct limit of $(L(A\cap \sigma_\a, \sigma_\a),j_{\sigma_\a, \sigma_\b}: \a<\b<\omega_1^{Ph]})$. It follows that $A\cap \bR^{P[h]}$ is $A$-full. 
\end{proof}

Fix a strong cardinal $\k$ of $\K$ and a $P$-generic $h\subseteq Coll(\omega, <\k)$. Set $\sigma=\sigma_h$. Working in $P[h]$, let $\mathcal{G}$ be the set of $\R$ such that $\R$ is a $\Lambda$-iterate of $\K$ via $\T$ such that $\lh(\T)<\kappa$, $\T$ is based on $\K|\k$ and $\pi^\T$ is defined (i.e, no drop on the main branch). We have that $\mathcal{G}$ is a directed set of iterates of $\K$ and hence, we can let $\M$ be the direct limit of $\mathcal{G}$.

\begin{claim}\label{nu is in hsigma} $\M\insegeq \mH_\sigma$.
\end{claim}
\begin{proof} Working in $V[g]$, we find a $\Sigma$-iteration $j:P\rightarrow N$ of $P$ of length $\omega_1$ such that $j(\k)=\omega_1$ and
$\bR$ is symmetrically generic over $N|\omega_1$.
This can be done by just enumerating $\bR$ in length $\omega_1$ and making the reals generic one by one and using the least normal measure with critical point $\kappa$ to stretch the sequence of Woodins. We now also have that $(j(\K)|\omega_1, \Psi)$ is a generic generator for $\Delta$ where $\Psi$ is the strategy of $j(\K)$ induced by $\Sigma_{N}$ (we verify this claim below). It then follows that $\M_\infty(j(\K)|\omega_1, \Psi)\insegeq \mathcal{H}$. As $\mH=j(\mH_\sigma)$\footnote{Here we use the fact that $\mH$ has a $\Delta^1_1(A, x_0)$ code.} and $\M_\infty(j(\K)|\omega_1, \Psi)=j(\M)$, the claim follows. Thus it remains to prove the following subclaim.

\begin{subclaim} $(j(\K)|\omega_1, \Psi)$ is a generic generator for $\Delta$
\end{subclaim}
\begin{proof} It is easy to verifying that $(j(\K)|\omega_1, \Psi)$ satisfies clauses 1-5 of \rdef{generic generator def}. For example, the club $C$ in clause 3 consist of stages where we used the measure to stretch our sequence of Woodin cardinals. Clause 5 follows from our choice of $\Delta$. We verify clause 6. 

It is enough to verify clause 6 for $A$ which have the form $Code(\Phi)$ where $(\S, \Phi)$ is a pure mouse pair. Let $(z_\a: \a<\omega_1)$ be the generic enumeration of the reals in $V[g]$ and let $(P_\a:\a<\omega_1)$ be the iterates of $P$ that our genericity iteration produces. Let $p_{\a, \b}:P_\a\rightarrow P_\b$ be the iteration embedding according to $\Sigma$. Because we are assuming $\neg {\sf{LEC}}$, we must have that $Code(\Phi)\in \Delta$, and therefore, there is $\a<\omega_1$ and $\b<\omega_1$ such that for some $\chi<p_{0, \b}(\k)$, $z_\a\in P_\b[h]$ where $h\subseteq Coll(\omega, \chi)$ is $P_\b$-generic. It follows that $(P_\b[h], \Sigma_{\P_\b})$ Suslin, co-Suslin capture $Code(\Phi)$. 

We now use universality of backgrounded constructions (see \cite[Chapter 2.3]{HMMSC} or \cite[Chapter 8.4]{SteelCom}). It follows from universality that for some $\tau$ smaller than the largest Woodin cardinal of $P_\b$ (call it $\d_\b$), the $L[\vec{E}]$ construction of $p_{0, \b}(\K)|\tau$ that is done using background extenders with critical points $>\chi$ reaches an iterate of $(\S, \Phi)$. Because $p_{0, \b}(\k)$ is $<\d_\b$-strong in $p_{0, \b}(\K)$, we can find such a $\tau$ below $p_{0, \b}(\k)$.\footnote{It is important that $p_{0, \b}(\k)$ is a $<\d_\b$-strong cardinal in $p_{0, \b}(\K)$ not just in $P_\b$. This is because to reflect $\tau$ below $p_{0, \b}(\k)$ we need to find an extender $E\in \vec{E}^{P_\b}$ with critical point $p_{0, \b}(\k)$ such that $\pi_E(p_{0, \b}(\K))|\tau=p_{0, \b}(\K)|\tau$ and $\lh(E)>\tau$. Because $p_{0, \b}(\k)$ is a $<\d_\b$-strong cardinal in $p_{0, \b}(\K)$, there are many such extenders.} It then follows that for some $\tau'<p_{0, \b}(\k)$, $Code(\Phi)$ is projective in $Code(\Psi_{p_{0, \b}(\K)|\tau'}$.
\end{proof}
\end{proof}

Suppose that $U\in P$ is such that $P\models ``U$ is a normal, $\k$-complete ultrafilter on $\k"$ and $U$ coheres $\Sigma$.  Then\\\\
(1) $\pi_U\rest \M$ is the iteration embedding according to $\pi_U(\Lambda\rest P)=\Phi\rest Ult(P, U)$ where $\Phi$ is the strategy of $\pi_U(\K)$ induced by $\Sigma_{Ult(P, U)}$.\\\\
Fix now $U$ as above and set $i=\pi_U$.  Let $h'\subseteq Coll(\omega, <i(\k))$ be generic over $Ult(P, U)$ such that $h=h'\cap Coll(\omega, <\k)$. Set $\tau=\bR^{Ult(P, U)[h']}$. It follows from the proof of \rcl{claim 1} that $\tau$ is $A$-full. Notice that $i$ lifts to $i^+:P[h]\rightarrow Ult(P, U)[h']$. We now have that $i^+\rest R_\sigma=k_{\sigma, \tau}$. Let $\nu={\sf{Ord}}\cap \M$.

\begin{claim} $\sup i[\nu]<i(\nu)$.
\end{claim}
\begin{proof} We have that $i(\M)$ is the direct limit of all countable $\Phi$-iterates of $i(\K)|i(\k)$ in $Ult(P, U)[h]$, where $\Phi$ is the strategy of $i(\K)$ induced by $\Sigma_{Ult(P, U)}$. Let $m:\K|\k\rightarrow \M$ be the iteration  embedding according to $\Lambda$. It follows that $\nu=\sup(m[\k])$ and hence, $i(\nu)=\sup i(m)[i(\k)]$. Because $i$ is discontinuous  at $\k$, we have that $i(\nu)>\sup \{ i(m(\a)): \a<\k\}$.
\end{proof}

Let $\T$ be the $\K|\k$-to-$\M$ tree according to $\Lambda$ and let $\M^+$ be the last model of $\T$ when we apply it to $i(\K)|i(\k)$\footnote{Notice that (1) implies that this makes sense.}. We now have that\\\\
(2.1) $i(\M)$ is a $\Phi_{\M^+}$-iterate of $\M^+$ as $\T$ is an iteration of $i(\K)|i(\k)$ that can be used in the direct limit converging to $i(\M)$, and\\
(2.2) If $n:\M^+\rightarrow i(\M)$ is the iteration embedding then $n\rest \M=k_{\sigma, \tau}$.\\\\
Because $k_{\sigma, \tau}\rest \M\in R_\tau$\footnote{This is a general fact about determinacy models. $k_{\sigma, \tau}\rest \M\in R^+_\tau$ because $k_{\sigma, \tau}:R_\sigma\rightarrow R_\tau$ is in $R^+_\tau$. Notice then that because $R_\tau$ and $R^+_\tau$ have the same reals, $k_{\sigma, \tau}\rest \M\in R_\tau$ (in fact if for some $\iota<\Theta^{R_\sigma}$, $a\in R_\tau^+$ is a countable subset of $L_\iota(A\cap \tau, \tau)$ then $a\in R_\tau$).} and $k_{\sigma, \tau}\rest \M$ is discontinuous at $\nu$,\footnote{It follows from \rcl{nu is in hsigma} that $\nu\leq {\sf{Ord}}\cap \mH_\sigma$. Since $H_\sigma\in \dom(k_{\sigma, \tau})$, $\nu\in \dom(k_{\sigma, \tau})$.} we can define $\powerset(\nu)\cap \M^+$ in $R_\tau$ from $k_{\sigma, \tau}\rest \M$ and $k_{\sigma, \tau}(\M)$. Indeed, the next claim verifies this.

\begin{claim}\label{computing powerset from k} $\powerset(\nu)\cap \M^+$ is definable in $R_\tau$ from $k_{\sigma, \tau}\rest \M$ and $j_{\sigma, \tau}(\M)$
\end{claim}
\begin{proof}
We have that $\M|\nu$ is just the transitive collapse of $k_{\sigma, \tau}[\M]$. Let $\zeta=\sup k_{\sigma, \tau}[\nu]$. Notice that $\nu$ is an inaccessible cardinal in $\M$ and $\zeta$ is an inaccessible cardinal in $k_{\sigma, \tau}(\M)$. Suppose now $E\subseteq \nu$. We then have that the following are equivalent:\\\\
(c) $E\in \M^+$.\\
(d)  For every $\iota<\nu$, $E\cap \iota \in \M$ and there is $\Q\insegeq k_{\sigma, \tau}(\M)$ and $F\in \Q$ such that $\rho_1(\Q)=\zeta$ and letting\begin{itemize}
\item $p$ be the first standard parameter of $\Q$, 
\item for $\a<\zeta$, $\Q_\a$ be the transitive collapse of the $\Sigma_1$ hull of $\Q$ with parameters from $\{p\}\cup \a$ and \item $q_\a: \Q_\a\rightarrow \Q$ be the inverse of the transitive collapse,
\end{itemize}
for some $\a_0<\nu$, for all $\a\in k_{\sigma, \tau}[\nu]$ such that $\a\geq k_{\sigma, \tau}(\a_0)$, $\Q_\a\in \rge(k_{\sigma, \tau})$, $q_\a^{-1}(F)\in \rge(k_{\sigma, \tau})$ and $k_{\sigma, \tau}^{-1}(q_\a^{-1}(F))\cap k_{\sigma, \tau}^{-1}(\a)=E\cap k_{\sigma, \tau}^{-1}(\a)$.\\\\
It is not hard to see that if $E\in \M^+$ then there is a pair $(\Q, F)$ with the above properties: $F=i^+(E)\cap \zeta$ and $\Q$ is the least initial segment of $k_{\sigma, \tau}(\M)$ with the property that $\rho_1(\Q)=\zeta$ and $F\in \Q$. Assume now that (d) holds. We can find some $\N\insegeq \M^+$ with the property that letting $l: \M^+\rightarrow i(\M)$ be the iteration embedding according to $\Phi_{\M^+}$,
\begin{itemize}
\item $\rho_1(\N)=\nu$,
\item letting $r$ be the first standard parameter of $\N$ and $\R$ be the transitive collapse of the $\Sigma_1$-hull of $l(\N)$ with parameters from $\{l(r)\}\cup \zeta$, $\Q\insegeq \R$.
\end{itemize}
We can find such an $\N$ because $i(\M)|\zeta$ is a $\Phi_\M$-iterate of $\M^+|\nu=\M$ (see (2.2)). For $\a<\nu$, let $\N_\a$ be the $\Sigma_1$-hull of $\N$ with parameters from $\{ r\}\cup \a$ and let $n_\a:\N_\a\rightarrow \N$ be the inverse of the transitive collapse. Pick $\a_0$ witnessing (d). We then have that for some $\a_1\in [\a_0, \nu)$ for all $\a\in [\a_1, \nu)$\footnote{We just need to find $\a_1$ large enough such that $\Q$ is definable over $\R$ from parameters in $\{p_1(\R)\}\cup k_{\sigma, \tau}(\a_1)$ where $p_1(\R)$ is the first  standard parameter of $\R$.}, $\Q_\a\insegeq k_{\sigma, \tau}(\N_\a)$. It now follows that we have some $E'\in \N$ such that for $\a\in [\a_0, \nu)$, $n_\a^{-1}(E')=k_{\sigma, \tau}^{-1}(q_\a^{-1}(F))$. It then follows that for every $\a\in [\a_1, \nu)$, $E'\cap \a=E\cap \a$. Hence, $E'=E$.
\end{proof}

We now have that in $R^+_\tau$, 
\begin{itemize}
    \item $k_{\sigma, \tau}(\M)$ is ordinal definable from $A\cap \tau$, $x_0$ and $\sigma$\footnote{Recall that $\mH$ is $\Delta^1_1(A, x_0)$}, 
    \item $k_{\sigma, \tau}\rest \M$ is ordinal definable from $\sigma$ and $A\cap \tau$.
\end{itemize} 
It follows that in $R^+_\tau$, $\powerset(\nu)\cap \M^+$ is ordinal definable from $A\cap \tau$, $x_0$ and $\sigma$. It then follows that if $\sigma'$ is $A$-full, there is $y\in B\cap \sigma'$ that codes $\sigma$ and there is $y'\in B\cap \tau$ that codes $\sigma'$ then\\\\
(A) in $R^+_{\sigma'}$, $\powerset(\nu)\cap \M^+$ is ordinal definable from $A\cap \sigma'$, $x_0$ and $\sigma$.\footnote{Fixing such a $\sigma'$, we have embeddings $k_{\sigma, \sigma'}: R_\sigma\rightarrow R_{\sigma'}$ and $k_{\sigma', \tau}: R_{\sigma'}\rightarrow R_\tau$. Because $k_{\sigma, \tau}=k_{\sigma', \tau}\circ k_{\sigma, \sigma'}$ and $k_{\sigma, \tau}\rest \M\in R_\tau$, we have that $k_{\sigma, \sigma'}\rest \M\in R_{\sigma'}$ and $k_{\sigma', \tau}(\k_{\sigma, \sigma'}\rest \M)=\k_{\sigma, \tau}\rest \M$. Because $k_{\sigma', \tau}(\sigma)=\sigma$, $k_{\sigma', \tau}(x_0)=x_0$, $k_{\sigma', \tau}(A\cap \sigma')=A\cap \tau$ and $k_{\sigma', \tau}^{-1}(\powerset(\nu)\cap \M^+)=\powerset(\nu)\cap \M^+$ (as $\powerset(\nu)\cap \M^+$ countable in $R_\tau$), (A) follows.}\\\\
Notice now because the set $C$ is in $\utilde{\Delta}_{\Gamma}$ and because of our choice of $x_1$ we have that\\\\
(B) $\powerset(\nu)\cap \M^+\in C_\Gamma(\sigma)$.\footnote{Let $T$ be the tree projecting to $U$ that comes from a $\Gamma$-scale. We then have that $C_\Gamma(\sigma)=\bR\cap L(T, \sigma)$. For example, see Apendix A of \cite{HMMSC}. Notice now that if $t\subseteq Coll(\omega, \sigma)$ is $L(T, \sigma)$-generic then by our choice of $\Gamma$, there is $\sigma'\in L(T, \sigma)[t]$ such that $\sigma'$ is $A$-full and there is $y\in B\cap \sigma'$ that codes $\sigma$. We assume now that $t\in Ult(P, U)[h']$. It then follows that in $L(T, \sigma)[t]$, whenever $\sigma'$ has the above properties, $\powerset(\nu)\cap \M^+$ is ordinal definable in $R_{\sigma'}$ from $A\cap \sigma'$, $x_0$ and $\sigma$. It is now not hard to turn what we have into a definition of $\powerset(\nu)\cap \M^+$ over $L(T, \sigma)$.}\\\\
We now have that $\powerset(\k)\cap \K\in C_{\Gamma}(P|\k, m)[h]$ as $C_\Gamma(\sigma)\subseteq C_{\Gamma}(P|\k, m)[h]$\footnote{ $m:\K\rightarrow \M$ is the iteration  embedding according to $\Lambda$. We can compute $\powerset(\k)\cap \K$ from $m$ and $\powerset(\nu)\cap \M^+$. This is similar to \rcl{computing powerset from k}.}. However, our calculations are independent of $h$, implying that in fact, $\powerset(\k)\cap \K\in C_{\Gamma}(P|\k, m)$.

Everything we have done above depended on $\k$. So we let $(\M_\k, m_\k)$ be $(\M, m)$ defined above. We thus have that $\K|(\k^+)^\K\in C_\Gamma(P|\k, m_\k)$.

To get a superstrong extender in $\K$, and hence a contradiction, we follow the general set up of the universality argument (see \cite[Chapter 2.3]{HMMSC}). Let $\Q$ be the direct limit of all $<\d$-iterates of $\K$ that are in $P$ and let $q: \K\rightarrow \Q$ be the iteration embedding according to $\Lambda$. Notice that $(\Q, q)$ is definable over $P|\d$. Also, because of our choice of $(P, \Sigma, \Gamma)$, $C_\Gamma(P|\d)\in P$ and $(\K, \Q, q)\in C_\Gamma(P|\d)$. 

Let now $\k$ be a strong cardinal of $\K$ and $\pi: N\rightarrow P|(\d^+)^\P$ be an elementary embedding such that $\cp(\pi)=\k$, and $(\Q, q)\in \rge(\pi)$. It follows that $\pi^{-1}(\Q, q)=(\M_\k, m_\k)$. Using condensation of sjs (see \cite[Lemma 2.19]{HMMSC}), we can find such a pair $(\pi, N)$ such that $C_\Gamma(P|\k, m_\k)\in N$. It then follows that $\K|(\k^+)^\K\in N$. 

Working in $P$, we now find a continuous chain $(X_\a: \a<\d)$ of elemenetary substructures of $P|(\d^+)^\P$ such that 
\begin{enumerate}
\item $X_\a$ has cardinality $<\d$,
\item for every $\a<\d$, $\k_\a=_{def}X_\a\cap \d\in \d$,
\item letting $N_\a$ be the transitive collapse of $X_\a$ and $\pi_\a: N_\a\rightarrow X_\a$ be the inverse of the transitive collapse, $C_\Gamma(P|\k_\a)\in N_\a$. 
\end{enumerate}
We can then find a $\k$ which witnesses Woodiness of $\d$ for $((X_\a: \a<\d), \K)$. For this $\k$, we have that $\K|(\k^+)^\K\in C_\Gamma(P|\k)\in N_\k$\footnote{For such $\k$, $m\in C_\Gamma(P|\k)$.}. We cam apply the argument used in \cite[Chapter 3, Chapter 4, Claim 4.1]{ESC} or the arguments used in \cite[Chapter 2.3]{HMMSC} to argue that $\k$ is a superstrong cardinal in $\K$.
\end{proof}

\begin{remark}
    \normalfont We remark that as it was the case with several other theorems proven in this paper, the corresponding version of \rthm{getting superstrongs} can also be established for hod pairs assuming $\neg {\sf{HPC}}$.
\end{remark}

\section{Concluding remarks}

The question that the reader must be asking is whether ${\sf{DLIH}}$ and ${\sf{BDL}}$ are at all plausible. Here we give some justifications. Recall Notation \ref{strong cardinals}.

\begin{definition}
    Suppose $F$ is as in \rthm{n*x}. We say $F$ is \textbf{pe-bad} if for a cone of $x\in \bR$, letting $\M$ be the $L[\vec{E}]$-construction of $\N^*_x|\d_x$, 
    $\sup(W(\M))=\d_x$. We define \textbf{hp-bad} similarly.  
\end{definition}

\begin{definition}
    Suppose $F$ is as in \rthm{n*x}. We let $Base(F)$ be the pointclass $\Gamma^*$ of \rthm{n*x}. We say $F$ is \textbf{projective-like} if for some set $C\subseteq \bR$ with $C\in\utilde{\Delta}_{Base(F)}$, $n\in [1, \omega)$ and a real $x$, $\Gamma^*=\Sigma^1_{2n}(C, x)$.
\end{definition}

\begin{conjecture}\label{important conjecture}
    Assume ${\sf{AD^+}}$. Suppose $F$ is as in \rthm{n*x} and is projective-like. Then $F$ is neither pe-bad nor hp-bad.
\end{conjecture}

It is not hard to show that ${\sf{HPC}}$ implies \rcon{important conjecture} for many projective $F$. Indeed, suppose $(\P, \Sigma)$ is a hod pair and $\Gamma^*$ is a good pointclass projective in $\Sigma$. Fix $F$ such that $Base(F)=\Gamma^*$ and suppose $x\in \dom(F)$ is such that $Code(\Sigma)$ is Suslin, co-Suslin captured by $(\N^*_x, \d_x, \Sigma_x)$. Let $\mH$ be the hod pair construction of $\N^*_x$ and let $\Lambda$ be the strategy of $\mH$ induced by $\Sigma_x$. We then have that for some $\xi$, $Code(\Sigma)$ is $\utilde{\Delta}^1_1(Code(\Lambda_{\mH|\xi}))$\footnote{This follows from the comparison argument of \cite{SteelCom}.}. It then follows that $\mH$ can have at most finitely many Woodin cardinals above $\xi$ as otherwise $\Gamma^*$ cannot be projective in $\Sigma$. 

Next we conjecture the following:
\begin{conjecture}\label{another con}
    Assume ${\sf{AD_{\bR}}}$. Suppose that for every set $B\subseteq \bR$ there is $F$ as in \rthm{n*x} such that $F$ is not pe-bad and $B\in \utilde{\Delta}_{Base(F)}$. 
    \begin{enumerate}
        \item (Pure Extender Version) Assume that ${\sf{LEC}}$ fails and set $\Delta=\Delta_{gen, pe}$. Then both ${\sf{DLIH}}(\Delta)$ and ${\sf{BDL}}(\Delta)$ hold.
        \item (Hod Pair Version) Assume that ${\sf{HPC}}$ fails and set $\Delta=\Delta_{gen, hp}$. Then both ${\sf{DLIH}}(\Delta)$ and ${\sf{BDL}}(\Delta)$ hold.
    \end{enumerate}
\end{conjecture}

The main motivation behind \rcon{another con} is the fact that we are able to prove it assuming ${\sf{NWLW}}$.  Indeed, assuming ${\sf{NWLW}}$, we are able to establish a comparison theorem for generic generators. Namely, if $(\P, \Sigma)$ and $(\Q, \Lambda)$ are generic generators for $\Delta$ such that $\sup(W(\P)\cup W(\Q))<\omega_1$ then there is a normal iterate $(\R, \Psi)$ of $(\P, \Sigma)$ and $(\Q, \Lambda)$ such that letting $i: L[\P]\rightarrow L[\R]$ and $j:L[\Q]\rightarrow L[\R]$ be the iteration embeddings, $i(\omega_1)=\omega_1=j(\omega_1)$. ${\sf{DLIH}}(\Delta)$ and ${\sf{BDL}}(\Delta)$ now follow from this comparison theorem and the proof of \cite[Theorem 3.8]{CCM}. All of this will appear in \cite{CWLD}. 
To prove \rcon{another con}, we need to prove a comparison theorem for generic generators without assuming ${\sf{NWLW}}$.

Our last remark is then to prove ${\sf{LEC}}$ or ${\sf{HPC}}$ it is enough to establish \rcon{important conjecture} and the following conjecture.

\begin{conjecture}
    Suppose $V\models {\sf{AD_{\mathbb{R}}}}+{\sf{NLE}}+V=L(\powerset(\bR))$. Suppose further that 
    \begin{itemize}
        \item $\Delta\subseteq \powerset(\bR)$,
        \item $g\subseteq Coll(\omega_1, \bR)$,
        \item in $V[g]$, $(\P, \Sigma)$ and $(\Q, \Lambda)$ are two generic generators for $\Delta$ of the same type\footnote{Either both are pure extender pairs or both are hod pairs.} such that $\sup(W(\P)\cup W(\Q))<\omega_1$.
    \end{itemize}
    Then, in $V[g]$, there is a common iterate $(\R, \Psi)$ of $(\P, \Sigma)$ and $(\Q, \Lambda)$ such that letting $i: L[\P]\rightarrow L[\R]$ and $j:L[\Q]\rightarrow L[\R]$ be the iteration embeddings, $i(\omega_1)=\omega_1=j(\omega_1)$
\end{conjecture}

\bibliographystyle{plain}
\bibliography{main}
\end{document}